\documentclass[12pt]{article}

\usepackage{amsthm}
\usepackage{amsmath}
\usepackage{amssymb}
\usepackage{ascmac}
\usepackage{mathrsfs}
\usepackage{graphicx}
\usepackage{longtable}
\usepackage{bm}
\usepackage{braket}
\usepackage{setspace}
\usepackage{mathrsfs}
\usepackage{subfig}
\usepackage[all]{xy}
\usepackage{here}
\usepackage[toc,page]{appendix}
\usepackage{caption}
\usepackage{hyperref}
\usepackage{comment}
\usepackage{multicol}
\usepackage{url}
\usepackage{fancyhdr}
\usepackage{ulem}

\usepackage{tikz}
\usetikzlibrary{arrows.meta,positioning, calc}
\usetikzlibrary{knots, braids, spath3, hobby}

\usepackage{diagbox}

\usepackage{enumitem}
\setlist[enumerate]{label=(\arabic*),font=\upshape}

\makeatletter
\def\mojiparline#1{
    \newcounter{mpl}
    \setcounter{mpl}{#1}
    \@tempdima=\linewidth
    \advance\@tempdima by-\value{mpl}zw
    \addtocounter{mpl}{-1}
    \divide\@tempdima by \value{mpl}
    \advance\kanjiskip by\@tempdima
    \advance\parindent by\@tempdima
}

\makeatother

\makeatother

\theoremstyle{definition}
\newtheorem{defi}{Definition}
\newtheorem{thm}[defi]{Theorem}
\newtheorem{cor}[defi]{Corollary}
\newtheorem{prop}[defi]{Proposition}
\newtheorem{lem}[defi]{Lemma}
\newtheorem{ex}[defi]{Example}
\newtheorem{rem}[defi]{Remark}

\newcommand{\ctext}[1]{\raise0.2ex\hbox{\textcircled{\scriptsize{#1}}}}

\newcommand{\NN}{\mathbb{N}}
\newcommand{\ZZ}{\mathbb{Z}}
\newcommand{\QQ}{\mathbb{Q}}

\title{HOMFLY Polynomials of the Torus Links $T(\pm 3,n)$}
\author{Norihisa Takahashi \thanks{Ryukoku University}}
\date{}

\begin{document}

\maketitle

\begin{abstract}
    We derive explicit formulas for the HOMFLY polynomials of the torus links $T(3,n)$ using braid groups and the skein relation.
    We first treat the case of $T(2,n)$ and then derive a five-term linear recurrence for an auxiliary sequence associated with $T(3,n)$.
    By solving this recurrence using a generating function, we obtain an explicit formula for the HOMFLY polynomial $P(T(3,n);y,z)$ of $T(3,n)$.
    The corresponding formula for $T(-3,n)$ is subsequently obtained from the mirror-image formula for the HOMFLY polynomial.
    As an application, we show that the HOMFLY polynomial distinguishes the links $T(3,n)$ within this family and distinguishes $T(3,n)$ from its mirror image for $n\geq 2$.
\end{abstract}
%---------------------
\section{Introduction}
%---------------------

    A link is the image of a smooth embedding of a finite disjoint union of copies of $S^1$ into the three-dimensional sphere $S^3$.
    Two links $L$ and $L'$ are said to be equivalent if they are related by an ambient isotopy, and in this case we write $L \cong L'$.

    The classification of links up to equivalence is one of the fundamental problems in knot theory, and various link invariants have been studied for this purpose.
    In particular, polynomial invariants provide effective tools for
    distinguishing links through explicit calculations on link diagrams.
    Among them, the HOMFLY polynomial is an invariant of oriented links
    defined by a skein relation.
    Under suitable substitutions of its variables, it specializes to both the Alexander polynomial and the Jones polynomial \cite{Kawauchi2015KnotTheory}.

    The HOMFLY polynomial was originally introduced as a two-variable
    polynomial invariant by Freyd, Yetter, Hoste, Lickorish, Millett, and Ocneanu \cite{FreydYetterHosteLickorishMillettOcneanu1985}.
    It was also studied independently by Przytycki and Traczyk as an
    invariant of Conway type \cite{PrzytyckiTraczyk1987}.
    Jones's construction using Hecke algebras and representations of braid groups provides a fundamental framework connecting closed braid representatives with link polynomials \cite{Jones1987Hecke}.
    From this viewpoint, it is natural to represent links as closures of braids and to study quantities that are invariant under Markov moves.

    Many general formulas are already known for the HOMFLY polynomials of torus knots and torus links.
    For example, Rosso and Jones established a general formula for torus-knot invariants arising from quantum groups.
    This formula is now known as the Rosso--Jones formula \cite{RossoJones1993}.
    Furthermore, Lin and Zheng described colored HOMFLY polynomials using characters of Hecke algebras and Schur polynomials, and obtained explicit formulas for torus knots and torus links \cite{LinZheng2010}.
    These results provide a powerful representation-theoretic framework for understanding HOMFLY-type invariants of torus knots and links.

    From the viewpoint of mathematical physics, formulas based on
    Chern--Simons theory and matrix models have also been studied.
    Brini, Eynard, and Marino introduced spectral curves associated with torus knots and links and showed that colored HOMFLY invariants are generated by topological recursion \cite{BriniEynardMarino2012}.
    In their framework, quantum colored invariants of torus knots are
    expressed as integrals over a Cartan subalgebra, or equivalently as
    matrix-model integrals.
    Thus, HOMFLY-type invariants of torus knots and links have been studied extensively from several perspectives, including quantum groups, Hecke algebras, Chern--Simons theory, and matrix models.

    In particular, the ordinary HOMFLY polynomial of $T(3,n)$ can also be obtained, in principle, as a specialization of such general formulas.
    Therefore, the aim of this paper is not to claim the existence of a new invariant or to replace the Rosso--Jones formula.
    Rather, our purpose is to give a direct skein-theoretic derivation specialized to the three-braid case.
    More precisely, we derive a recurrence relation for $T(3,n)$ from braid calculations, Markov moves, and the skein relation, and then solve this recurrence explicitly.
    The resulting formula is written in the same variables $y,z$ as those used in the skein relation, and hence it can be applied directly to concrete computations and specializations.
    One advantage of this skein-theoretic formula is that it is expressed directly in the variables $y,z$ used in the defining skein relation.
    This makes it possible, for example, to read off the highest degree in $z$ and to distinguish the links $T(3,n)$ within this family, as shown in Section~\ref{sec:applications}.

    The contribution of this paper lies in an elementary and explicit derivation of the ordinary HOMFLY polynomials of $T(\pm 3,n)$.
    As a preliminary step, we first treat the case of $T(2,n)$.
    We then analyze the five-term recurrence that arises in the computation of $T(3,n)$ and express its solution in terms of an auxiliary sequence $G_n$ defined by a generating function.
    A characteristic feature of the final formula is that the coefficient of the possible $G_{n+1}$ term cancels.
    Finally, the case of $T(-3,n)$ is obtained from the mirror-image formula for the HOMFLY polynomial.
%----------------------
\section{Preliminaries}
%----------------------
%-------------------------------------
\subsection{Braid groups and closures}
%-------------------------------------

    In this section, we recall the basic facts needed to represent links as closures of braids.

    Equip the plane with its standard orientation.
    At a crossing of an oriented link diagram, let the first vector be the tangent vector to the overpassing arc in the direction of its orientation, and let the second vector be the tangent vector to the underpassing arc in the direction of its orientation.
    The crossing is called positive if this ordered pair agrees with the orientation of the plane, and negative otherwise.

    Fix $n$ distinct points $p_1,\dots,p_n$ in $D^2$.
    A geometric braid on $n$ strands is an embedding of $n$ mutually disjoint arcs in $D^2\times[0,1]$ such that each arc joins a point $(p_i,0)$ in $D^2\times\{0\}$ to a point $(p_j,1)$ in $D^2\times\{1\}$, and the projection onto $[0,1]$ is monotone on each arc.
    We identify geometric braids up to isotopy fixing their endpoints.
    By defining the product of two braids by stacking one on top of the other, the set of equivalence classes of braids on $n$ strands forms a group.
    This group is called the braid group on $n$ strands and is denoted by $B_n$.

    \begin{thm}[Artin's presentation {\cite[Theorem 1.8]{Birman1974BraidsLinksMCG}}]
        The braid group $B_n$ admits the presentation
            \[
                B_n =
                \left\langle
                    \sigma_1,\sigma_2,\dots,\sigma_{n-1}
                    \,\middle|\,
                    \begin{matrix}
                    \sigma_i\sigma_{i+1}\sigma_i
                    =
                    \sigma_{i+1}\sigma_i\sigma_{i+1}
                    & (i=1,2,\dots,n-2),\\
                    \sigma_i\sigma_j=\sigma_j\sigma_i
                    & (|i-j|>1)
                    \end{matrix}
                \right\rangle .
            \]
        Here, $\sigma_i$ is the elementary braid in which the $i$th and
        $(i+1)$st strands form a positive crossing.
    \end{thm}

    \begin{figure}[ht]
        \centering
        \begin{tabular}{cc}
        \begin{minipage}{0.45\textwidth}
            \centering
            \captionsetup{width=0.8\linewidth}
            \includegraphics[width=\textwidth]{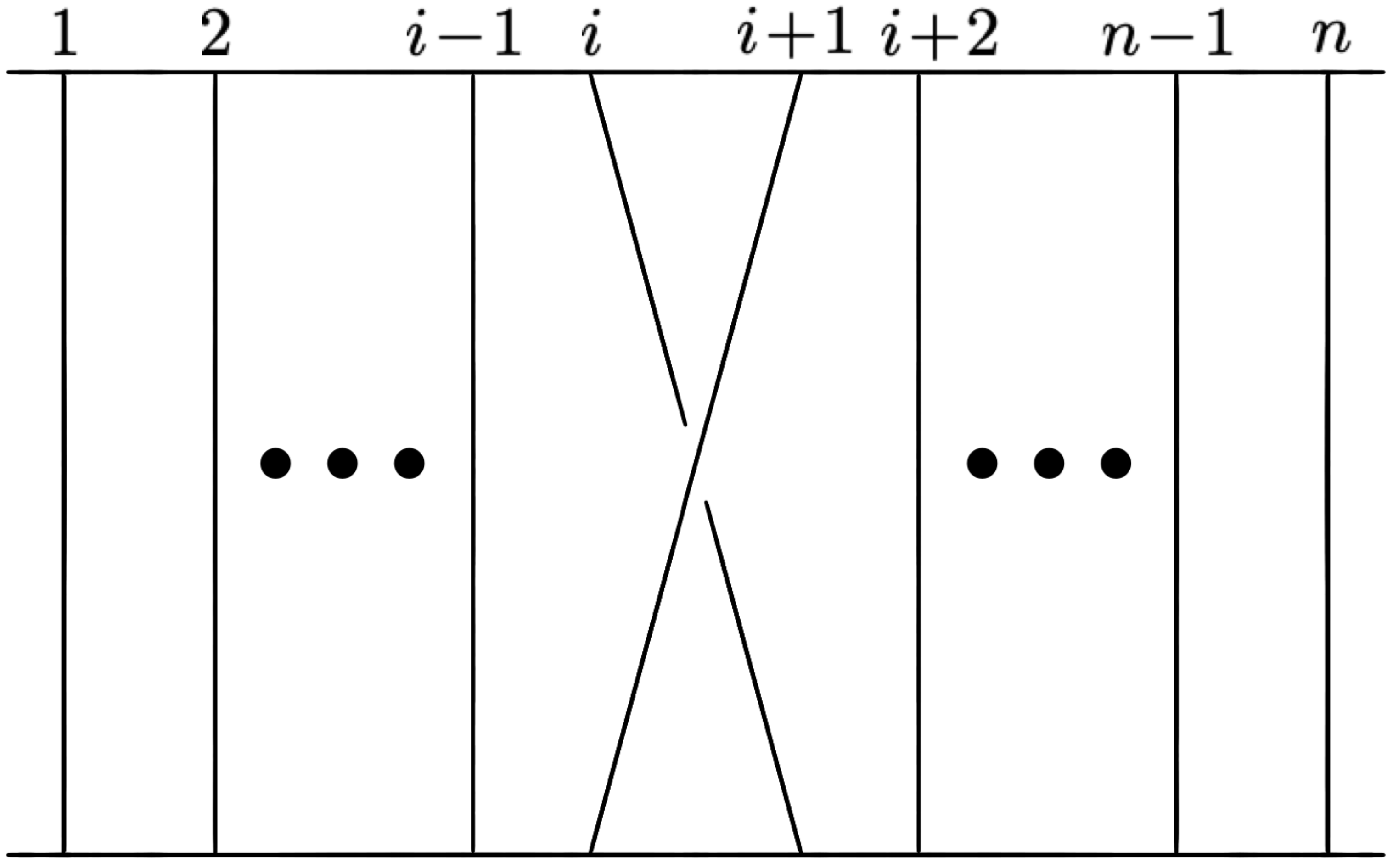}
            \caption{$\sigma_i$}
            \label{fig:sigmai}
        \end{minipage}
        \hspace{5mm}
        \begin{minipage}{0.45\textwidth}
            \centering
            \captionsetup{width=0.8\linewidth}
            \includegraphics[width=\textwidth]{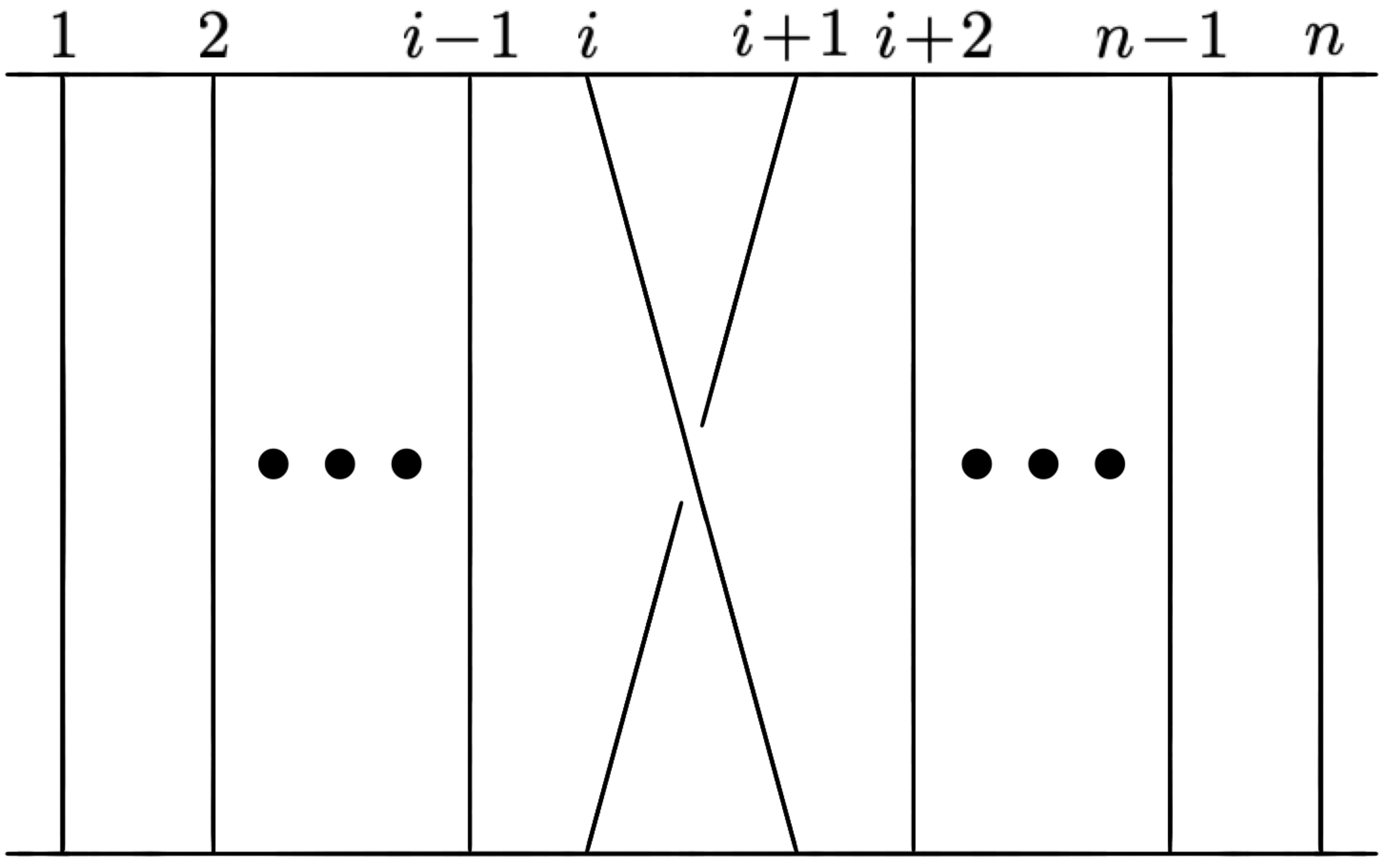}
            \caption{$\sigma_i^{-1}$}
            \label{fig:sigmai_inv}
        \end{minipage}
        \end{tabular}
    \end{figure}

    For $x\in B_n$, connect the $k$th endpoint at the top of a diagram of $x$ to the $k$th endpoint at the bottom, as shown in Figure~\ref{fig:closure}.
    The resulting link is called the closure of $x$ and is denoted by
    $\overline{x}$.
    If a link $L$ satisfies $L=\overline{x}$, then $\overline{x}$ is called a braid representative of $L$.

    \begin{figure}[ht]
        \centering
        \includegraphics[width=0.5\textwidth]{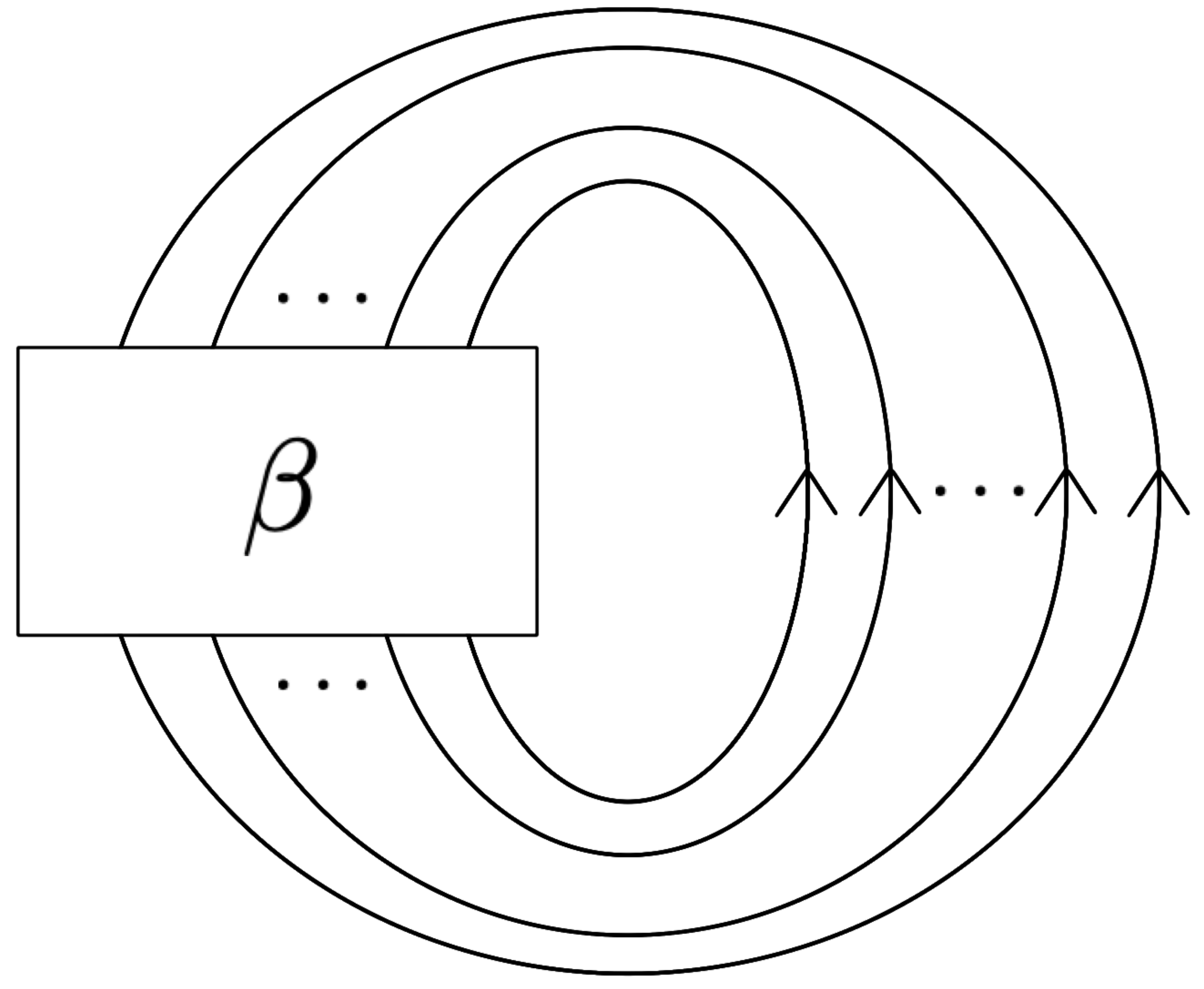}
        \caption{The closure of a braid}
        \label{fig:closure}
    \end{figure}

    The following theorem states that every link can be represented as the closure of a braid.

    \begin{thm}[Alexander's theorem {\cite[Theorem 2.1]{Birman1974BraidsLinksMCG}}]
        For every link $L$, there exist an integer $n\in\NN$ and a braid $x\in B_n$ such that $L=\overline{x}$.
    \end{thm}

    In general, a braid representative of a link is not unique.
    We write $L\cong L'$ when two links $L$ and $L'$ are equivalent.
    The following theorem describes the two fundamental operations relating braid representatives with equivalent closures.

    \begin{thm}[{Markov's theorem \cite{Markov1936FreeEquivalenceClosedBraids}}]
    \label{thm:Markov}
        The closures of two braids are equivalent if and only if one braid can be obtained from the other by a finite sequence of the following two operations:
        \begin{enumerate}
            \item For $x,y\in B_n$, $\overline{x}\cong\overline{y^{-1}xy}$.
            \item For $x\in B_n$, $\overline{x} \cong \overline{\iota(x)\sigma_n^{\pm1}}$, where $\iota:B_n\hookrightarrow B_{n+1}$ is the natural inclusion.
        \end{enumerate}
    \end{thm}

    \begin{rem}
        The operation in Theorem~\ref{thm:Markov}(2) is called
        Markov stabilization.
        It increases the number of strands by one.
        The inverse operation is called Markov destabilization.
        The original form of Markov's theorem is due to Markov; for a precise formulation and proof of the equivalence of closed braids, see also Corollary 2.3.1 of \cite{Birman1974BraidsLinksMCG}.
    \end{rem}

    \begin{figure}[ht]
        \centering
        \begin{tabular}{cc}
        \begin{minipage}{0.4\textwidth}
            \centering
            \captionsetup{width=0.8\linewidth}
            \includegraphics[width=\textwidth]{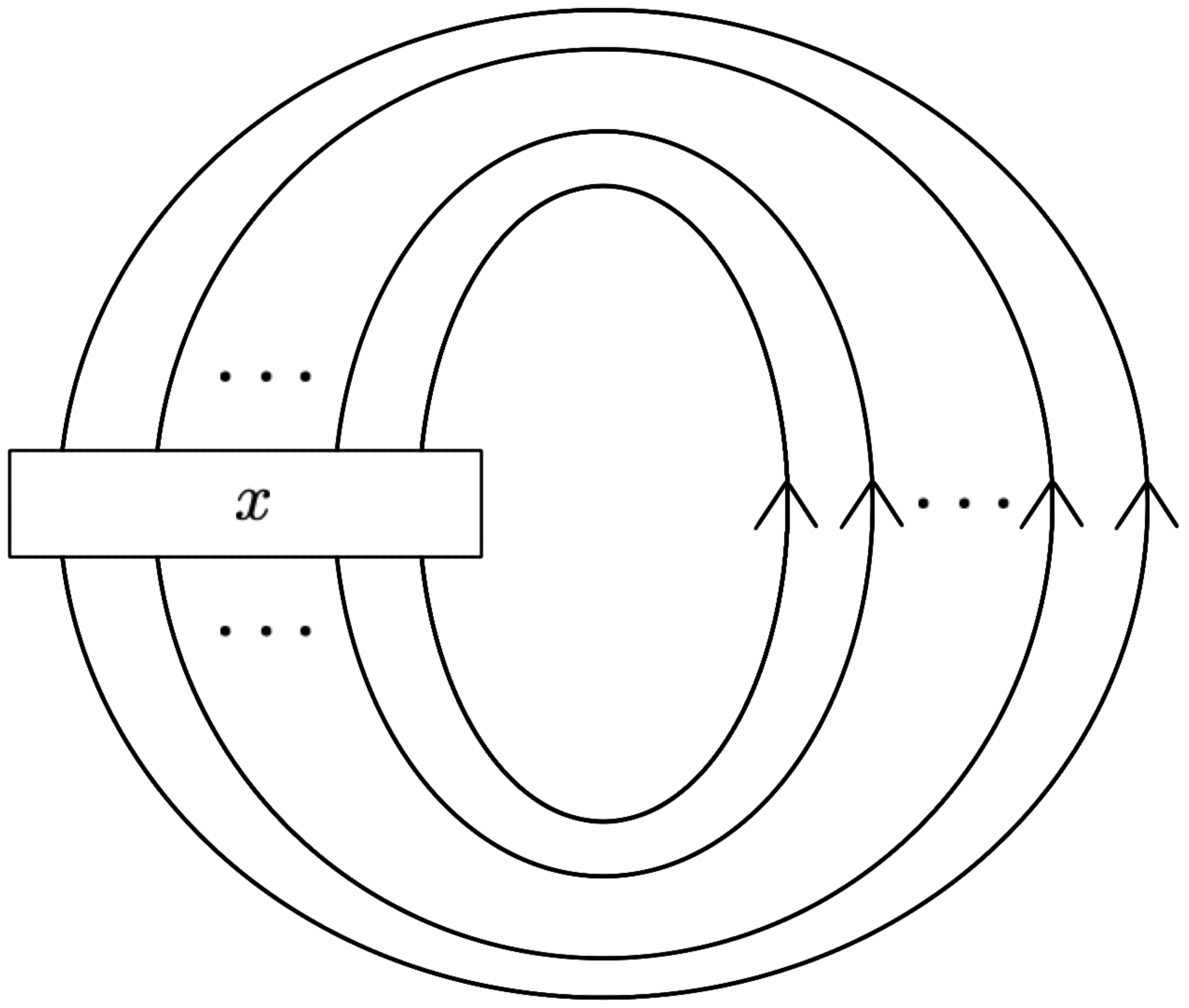}
        \end{minipage}
        $\cong$
        \begin{minipage}{0.4\textwidth}
            \centering
            \captionsetup{width=0.8\linewidth}
            \includegraphics[width=\textwidth]{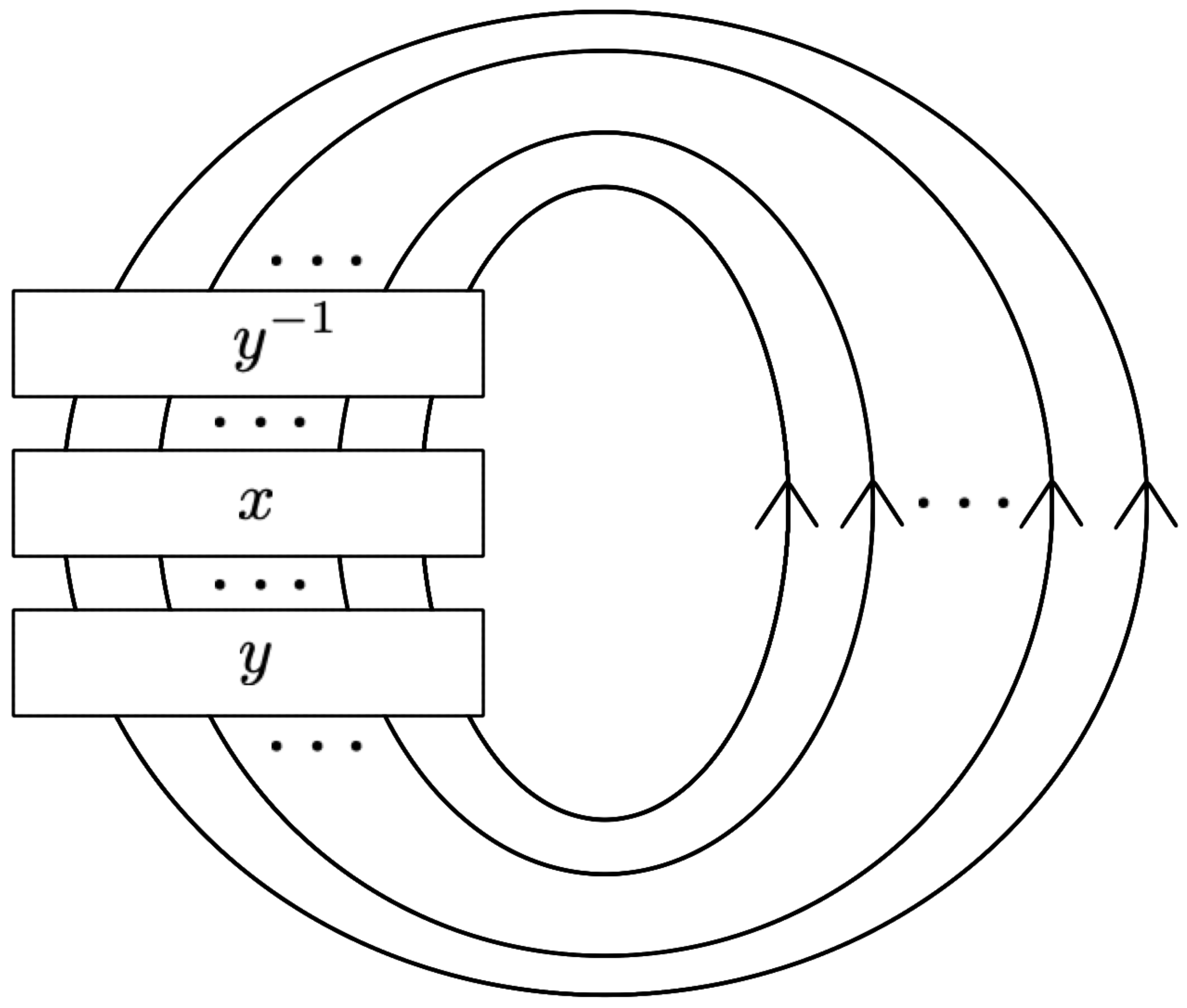}
        \end{minipage}
        \end{tabular}
        \caption{Markov move I}
        \label{fig:Markov_move_1}
    \end{figure}

    \begin{figure}[ht]
        \centering
        \begin{tabular}{ccc}
        \begin{minipage}{0.3\textwidth}
            \centering
            \captionsetup{width=0.8\linewidth}
            \includegraphics[width=\textwidth]{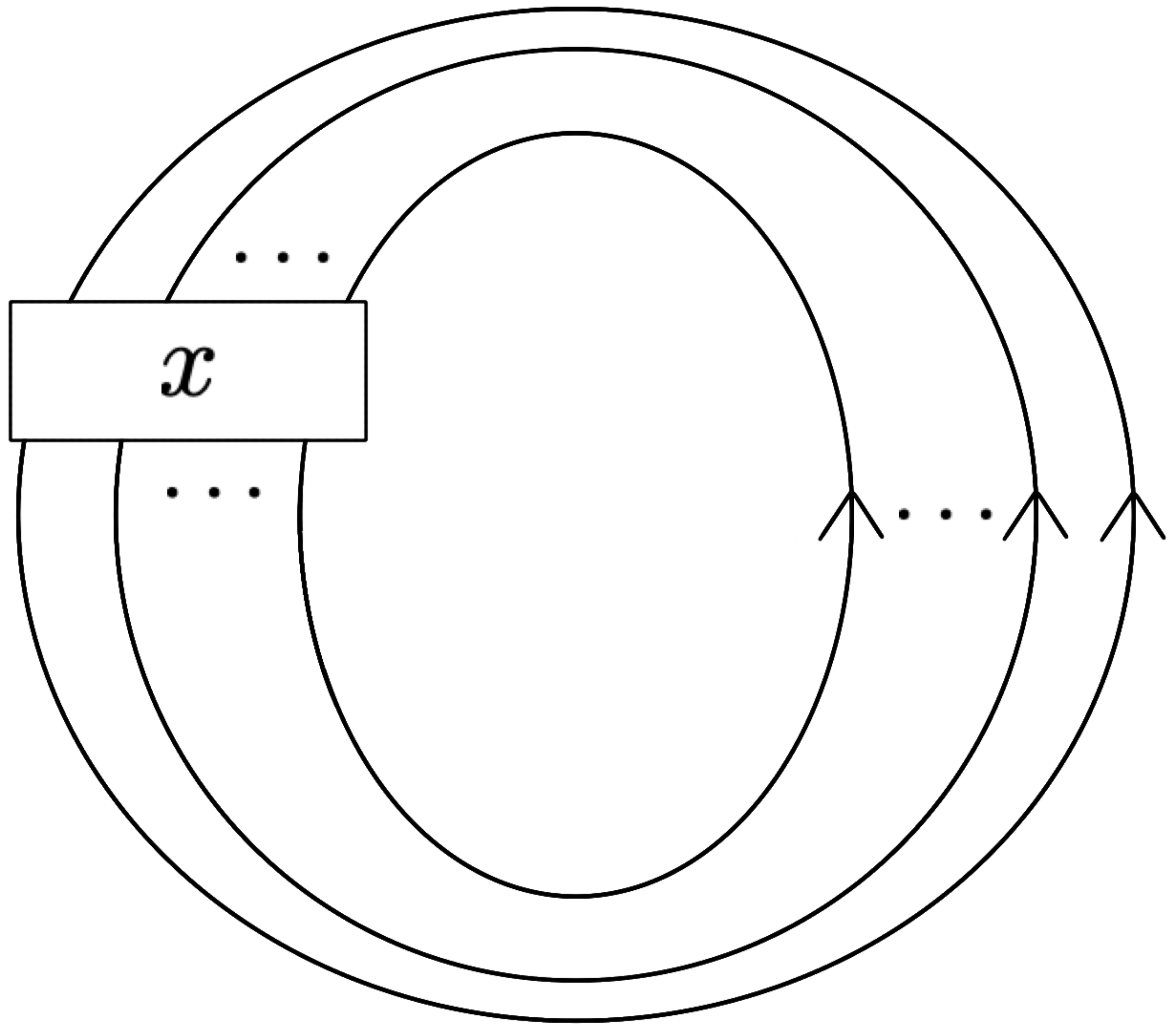}
        \end{minipage}
        $\cong$
        \begin{minipage}{0.3\textwidth}
            \centering
            \captionsetup{width=0.8\linewidth}
            \includegraphics[width=\textwidth]{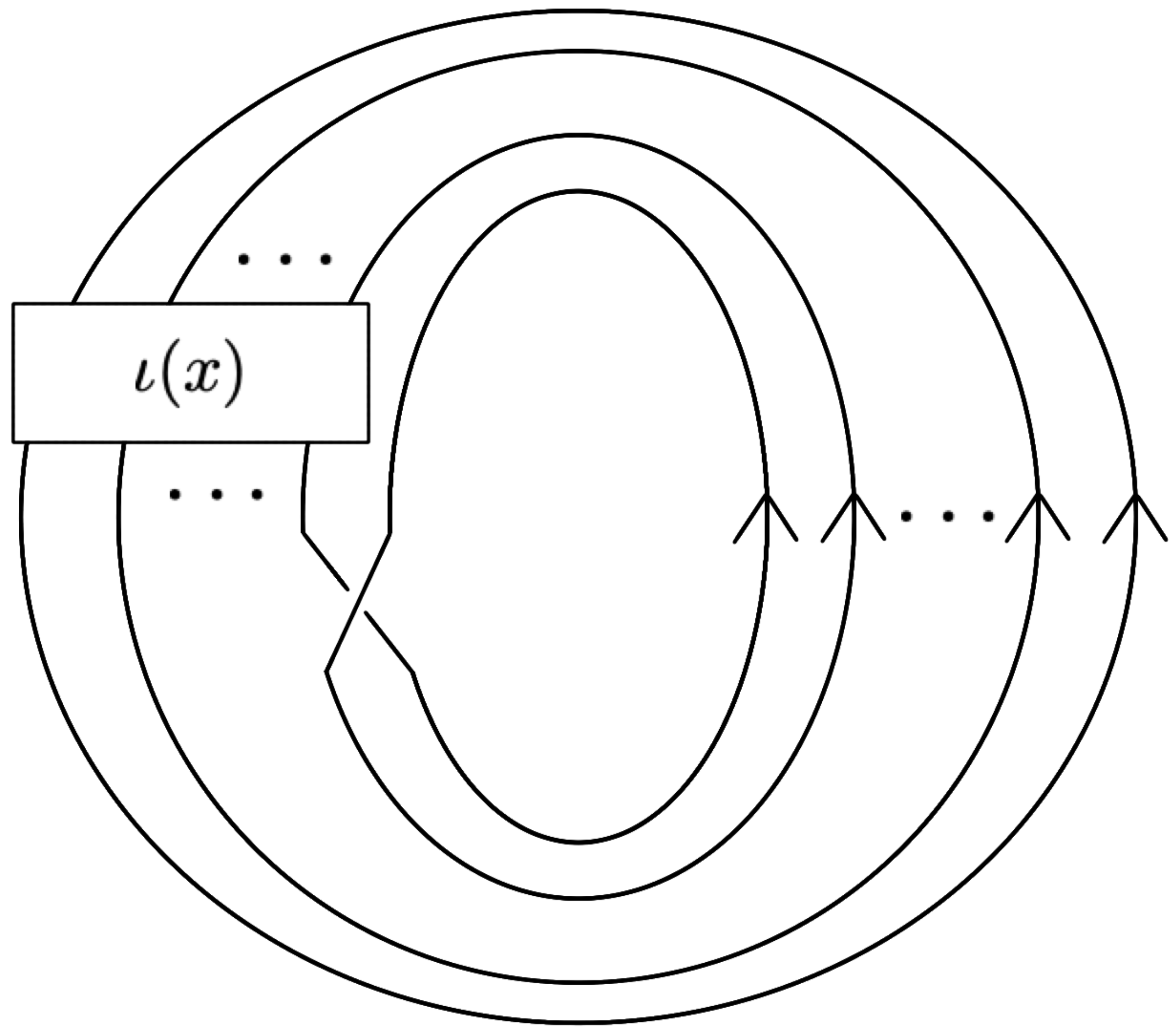}
        \end{minipage}
        $\cong$
        \begin{minipage}{0.3\textwidth}
            \centering
            \captionsetup{width=0.8\linewidth}
            \includegraphics[width=\textwidth]{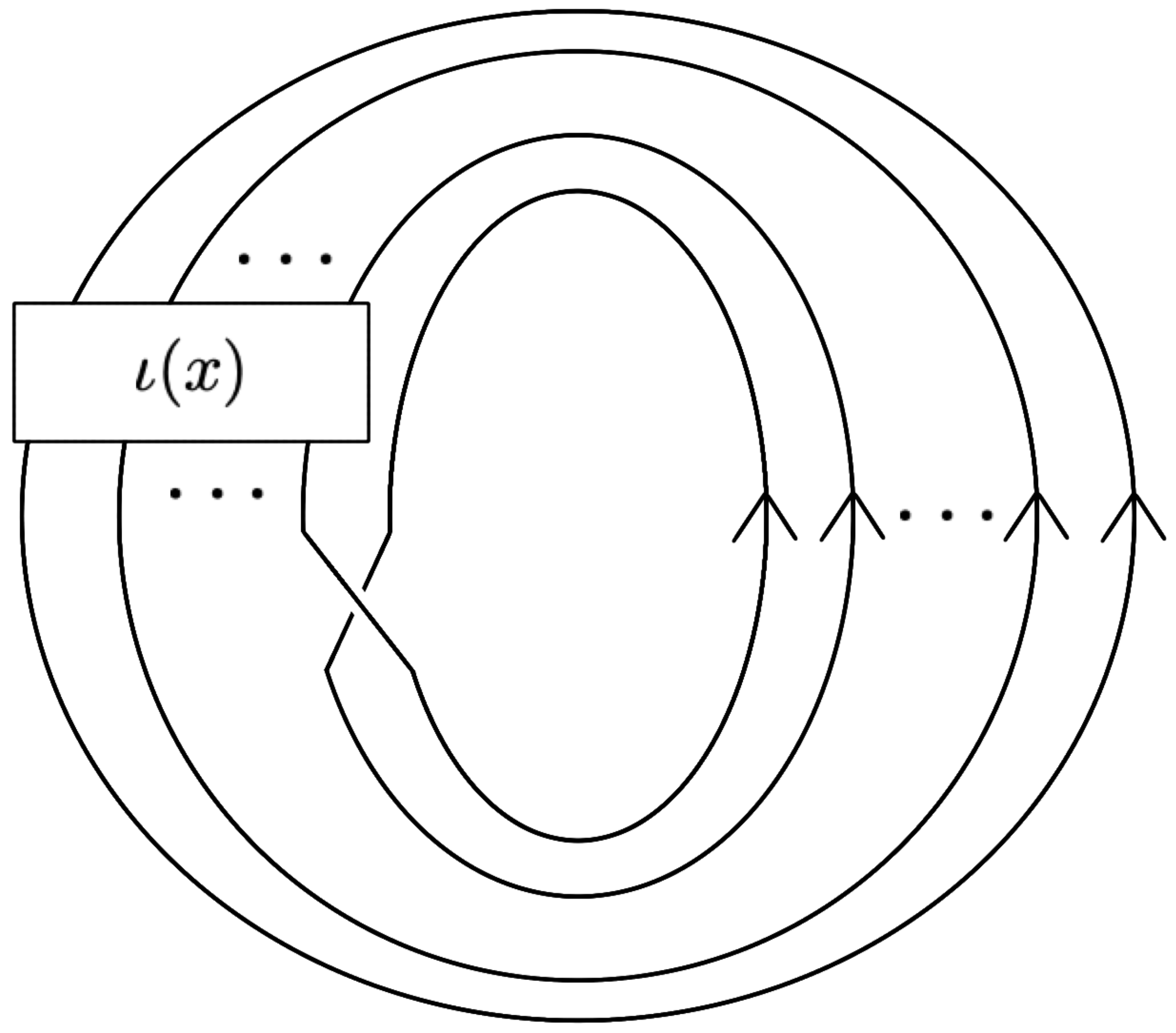}
        \end{minipage}
        \end{tabular}
        \caption{Markov move II}
        \label{fig:Markov_move_2}
    \end{figure}

    The first Markov move implies that, after taking closures, a cyclic permutation of a braid word does not change the resulting link.
    We state this fact as a lemma for later use.

    \begin{lem}
    \label{lem:cyclic_closure}
        For $x,y\in B_n$, the closures $\overline{xy}$ and
        $\overline{yx}$ are equivalent.
    \end{lem}

    \begin{proof}
        Applying Theorem~\ref{thm:Markov}(1) to the braids $xy$ and $x$, we obtain
            \[
                \overline{xy}
                \cong
                \overline{x^{-1}(xy)x}
                =
                \overline{yx}.
            \]
    \end{proof}

%---------------------------------
\subsection{The HOMFLY polynomial}
%---------------------------------

    In this section, we define the HOMFLY polynomial.
    We regard the HOMFLY polynomial as an element of the Laurent polynomial ring $\ZZ[y^{\pm1},z^{\pm1}]$.

    \begin{defi}[Skein triple]
        Let $D$ be an oriented link diagram and let $P$ be one of its
        crossings.
        Let $D_+$ and $D_-$ be the diagrams obtained by replacing the
        crossing at $P$ with a positive and a negative crossing, respectively, and let $D_0$ be the diagram obtained by smoothing the crossing.
        The triple $(D_+,D_-,D_0)$ is called a skein triple.
    \end{defi}

    \begin{figure}[ht]
        \centering
        \includegraphics[width=0.5\textwidth]{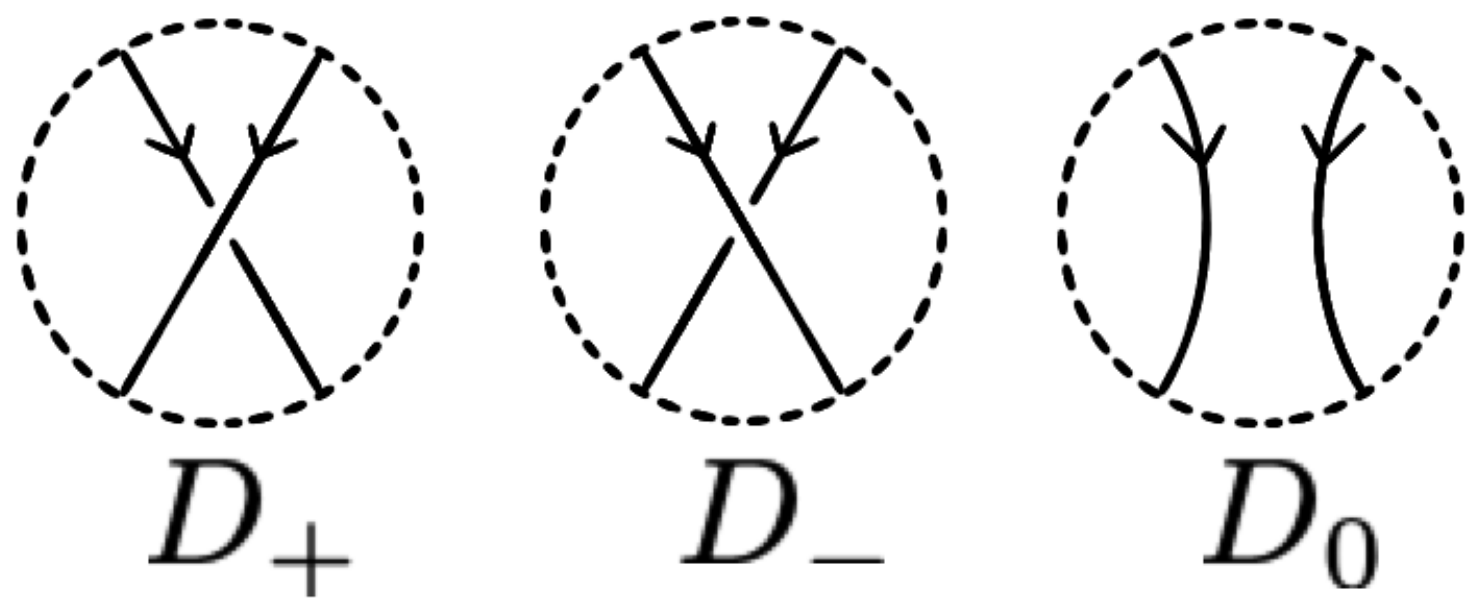}
        \caption{A skein triple}
        \label{fig:Skein_triple}
    \end{figure}

    We denote the unknot by $\bigcirc$.

    \begin{defi}[HOMFLY polynomial]
        Let $y$ and $z$ be indeterminates.
        The HOMFLY polynomial $P(D;y,z)$ of an oriented link diagram $D$ is characterized by the following two conditions:
        \begin{enumerate}
            \setlength{\itemsep}{5pt}
            \item $P(\bigcirc;y,z)=1$.
            \item For every skein triple $(D_+,D_-,D_0)$,
                \[
                    yP(D_+;y,z)
                    +y^{-1}P(D_-;y,z)
                    =
                    zP(D_0;y,z).
                \]
        \end{enumerate}
    \end{defi}

    Our definition follows \cite[Section 3]{Kawauchi2015KnotTheory}.
    Although various conventions for the HOMFLY polynomial appear in the literature, they are mutually equivalent after suitable changes of variables.

    The following theorem shows that $P(D;y,z)$ defines an invariant of
    oriented links.

    \begin{thm}[Existence and invariance of the HOMFLY polynomial {\cite{FreydYetterHosteLickorishMillettOcneanu1985, PrzytyckiTraczyk1987}}]
    \label{thm:HOMFLY_existence}
        There exists a unique Laurent polynomial $P(D;y,z)$ satisfying the two conditions above, and it is invariant under the Reidemeister moves.
    \end{thm}

    For an oriented link $L$, we define $P(L;y,z)$ to be the HOMFLY polynomial of any diagram representing $L$.

    The following proposition is an immediate consequence of the skein
    relation.

    \begin{prop}
    \label{prop:unknot_union}
        For any oriented link diagram $D$,
        \[
            P(\bigcirc\cup D;y,z)
            =
            z^{-1}(y+y^{-1})P(D;y,z).
        \]
    \end{prop}

    \begin{proof}
        Applying the skein relation to the local deformation shown in
        Figure~\ref{fig:Dhenkei}, we obtain
            \begin{align*}
                yP(D;y,z)+y^{-1}P(D;y,z)
                    &=zP(\bigcirc\cup D;y,z),\\
                P(\bigcirc\cup D;y,z)
                    &=z^{-1}(y+y^{-1})P(D;y,z).
            \end{align*}
    \end{proof}

    \begin{figure}[ht]
        \centering
        \includegraphics[width=0.5\textwidth]{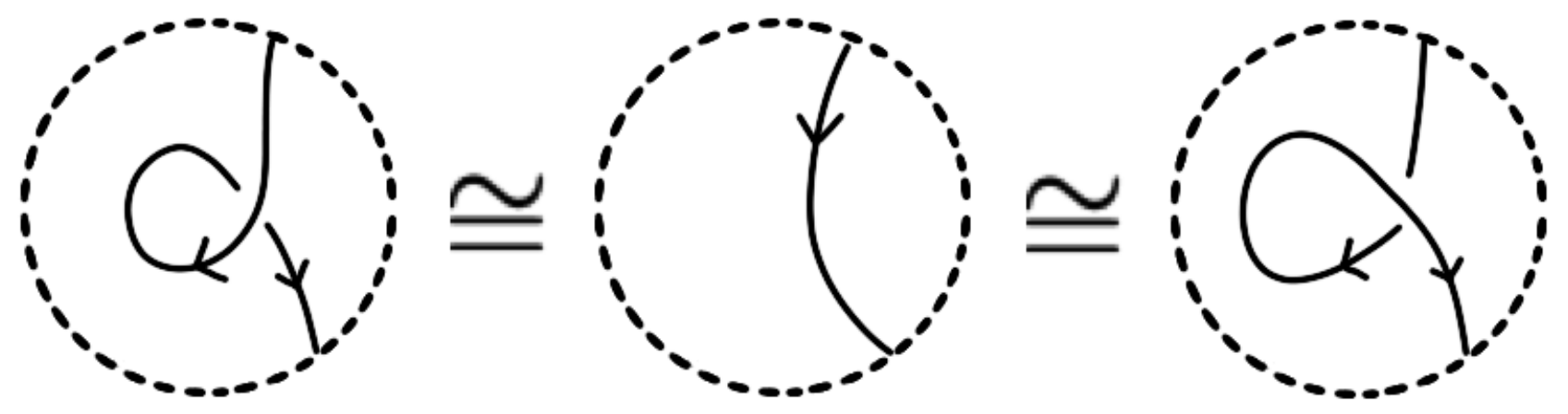}
        \caption{A local deformation of $D$}
        \label{fig:Dhenkei}
    \end{figure}

    The following theorem is called the mirror-image formula.

    \begin{thm}[Mirror-image formula {\cite[Proposition 16.1]{Lickorish1997Introduction}}]
    \label{thm:mirror_HOMFLY}
        Let $L^*$ denote the mirror image of an oriented link $L$.
        Then
        \[
            P(L^*;y,z)=P(L;y^{-1},z).
        \]
    \end{thm}

    \begin{proof}
        For an oriented link diagram $D$, let $D^*$ denote its mirror image and define
            \[
                Q(D;y,z):=P(D^*;y^{-1},z).
            \]
        Then $Q(\bigcirc;y,z)=1$.

        Taking the mirror image of a skein triple $(D_+,D_-,D_0)$ interchanges the positive and negative crossings.
        Hence,
            \[
                (D_+)^*=(D^*)_-,
                \qquad
                (D_-)^*=(D^*)_+,
                \qquad
                (D_0)^*=(D^*)_0.
            \]
        It follows from the skein relation for $D^*$ that
            \[
                yQ(D_+;y,z)
                +y^{-1}Q(D_-;y,z)
                =
                zQ(D_0;y,z).
            \]
        By the uniqueness in Theorem~\ref{thm:HOMFLY_existence}, we have $Q(D;y,z)=P(D;y,z)$, and the result follows.
    \end{proof}
%----------------------------------------------
\section{The HOMFLY polynomials of torus links}
%----------------------------------------------
%-------------------------------------
\subsection{Definition of torus links}
%-------------------------------------

    \begin{defi}[Torus link]
        For $m\in\ZZ_{\geq2}$ and $n\in\ZZ$, define
            \[
                T(m,n) := \overline{(\sigma_1\sigma_2\cdots\sigma_{m-1})^n}.
            \]
        This link is called the $(m,n)$-torus link.

        We further define $T(-m,n)$ to be the mirror image of $T(m,n)$.
        More precisely, $T(-m,n)$ is the closure of the oriented diagram obtained by reversing every crossing in the standard closed braid diagram $(\sigma_1\cdots\sigma_{m-1})^n$.
        The orientation of $T(-m,n)$ is the one induced from the standard closed braid diagram of $T(m,n)$ under this mirror operation.
        In other words, the components of $T(-m,n)$ are oriented so as to correspond to the mirrored components of $T(m,n)$.
        Throughout this paper, equivalence of such links means equivalence as oriented links with these prescribed orientations.
    \end{defi}

    \begin{rem}[{\cite[Section 5.1]{Adams2004KnotBook}}]
        If $n\neq0$, then $T(m,n)$ has $\gcd(m,|n|)$ components.
        In particular, $T(m,n)$ is a knot if $\gcd(m,|n|)=1$.
        If $n=0$, then $T(m,0)$ is the trivial link with $m$ components.
    \end{rem}

    \begin{thm}[{\cite[Section 5.1]{Adams2004KnotBook}}]
    \label{thm:torus_symmetry}
        For $m,n\geq2$, after choosing compatible orientations, we have
            \[
                T(m,n)\cong T(n,m).
            \]
    \end{thm}

%------------------------------------------------
\subsection{\texorpdfstring{The HOMFLY polynomials of $T(2,\pm n)$}{The HOMFLY polynomials of T(2,+/-n)}}
%------------------------------------------------

    We begin with the following elementary lemma.

    \begin{lem}
    \label{lem:3koukan}
        Let $K$ be a field of characteristic zero, and consider the quadratic equation
            \[
                X^2-pX+q=0
            \]
        over $K$.
        Let $\Omega$ be its splitting field and let $\lambda_1,\lambda_2\in\Omega$ be its roots.
        If $\lambda_1\neq\lambda_2$, then the general solution of the
        three-term recurrence
            \[
                a_{n+2}-pa_{n+1}+qa_n=0
            \]
        is
            \[
                a_n
                =
                \frac{\lambda_1^n-\lambda_2^n}
                     {\lambda_1-\lambda_2}a_1
                -
                \lambda_1\lambda_2
                \frac{\lambda_1^{n-1}-\lambda_2^{n-1}}
                     {\lambda_1-\lambda_2}a_0.
            \]
    \end{lem}

    \begin{proof}
        We perform the following calculations over $\Omega$.
        Using $\lambda_1$ and $\lambda_2$, the recurrence can be rewritten as
            \[
                \begin{matrix}
                    a_{n+2}-\lambda_1a_{n+1}
                    =
                    \lambda_2(a_{n+1}-\lambda_1a_n),\\
                    a_{n+2}-\lambda_2a_{n+1}
                    =
                    \lambda_1(a_{n+1}-\lambda_2a_n).
                \end{matrix}
            \]
        It follows that
            \[
                \begin{matrix}
                    a_{n+1}-\lambda_1a_n
                    =
                    \lambda_2^n(a_1-\lambda_1a_0),\\
                    a_{n+1}-\lambda_2a_n
                    =
                    \lambda_1^n(a_1-\lambda_2a_0).
                \end{matrix}
            \]
        Subtracting these two equations gives
            \[
                a_n
                =
                \frac{\lambda_1^n-\lambda_2^n}
                     {\lambda_1-\lambda_2}a_1
                -
                \lambda_1\lambda_2
                \frac{\lambda_1^{n-1}-\lambda_2^{n-1}}
                     {\lambda_1-\lambda_2}a_0.
            \]
        The right-hand side is symmetric in $\lambda_1$ and $\lambda_2$, and hence can be expressed in terms of $p$ and $q$.
        It may therefore be regarded as an expression over $K$.
    \end{proof}

    Consider the polynomial
        \[
            h_n(\lambda_1,\lambda_2) = \sum_{r=0}^n\lambda_1^{n-r}\lambda_2^r.
        \]
    Since $h_n$ is symmetric in $\lambda_1$ and $\lambda_2$, it can be
    expressed in terms of the elementary symmetric polynomials
        \[
            e_1=\lambda_1+\lambda_2,
            \qquad
            e_2=\lambda_1\lambda_2
        \]
    in the form
        \[
            h_n(\lambda_1,\lambda_2)
            =
            \sum_{p,q\geq0}C_{p,q}e_1^pe_2^q.
        \]

    Since $h_n$, $e_1$, and $e_2$ are homogeneous of degrees $n$, $1$, and $2$, respectively, we have $C_{p,q}=0$ whenever $p+2q\neq n$.
    Therefore,
        \[
            h_n(\lambda_1,\lambda_2)
            =
            \sum_{p+2q=n}C_{p,q}e_1^pe_2^q
            =
            \sum_{q=0}^{\left\lfloor \frac{n}{2}\right\rfloor}
            C'_{n,q}e_1^{n-2q}e_2^q.
        \]

    We now write
        \[
            H_n(e_1,e_2)
            :=
            h_n(\lambda_1,\lambda_2)
            =
            \sum_{j=0}^{\left\lfloor \frac{n}{2}\right\rfloor}
            C_{n,j}e_1^{n-2j}e_2^j
        \]
    and determine the coefficients $C_{n,j}$.

    First,
        \[
            H_0(e_1,e_2)=1,
            \qquad
            H_1(e_1,e_2)=e_1,
        \]
    and hence
        \[
            C_{0,0}=C_{1,0}=1.
        \]

    Moreover,
        \begin{align*}
            \lambda_1^{n+2}-\lambda_2^{n+2}
            &=
            (\lambda_1+\lambda_2)
            (\lambda_1^{n+1}-\lambda_2^{n+1})
            -
            \lambda_1\lambda_2
            (\lambda_1^n-\lambda_2^n)\\
            &=
            e_1(\lambda_1^{n+1}-\lambda_2^{n+1})
            -
            e_2(\lambda_1^n-\lambda_2^n).
        \end{align*}
    It follows that $H_n$ satisfies
        \[
            H_{n+2}(e_1,e_2) - e_1H_{n+1}(e_1,e_2) + e_2H_n(e_1,e_2) = 0.
        \]

    Therefore,
        \begin{align*}
            &H_{n+2}(e_1,e_2) - e_1H_{n+1}(e_1,e_2) + e_2H_n(e_1,e_2)\\
            &=
            \sum_{j=0}^{\left\lfloor \frac{n+2}{2}\right\rfloor}
            C_{n+2,j}e_1^{n+2-2j}e_2^j
            -
            \sum_{j=0}^{\left\lfloor \frac{n+1}{2}\right\rfloor}
            C_{n+1,j}e_1^{n+2-2j}e_2^j
            +
            \sum_{j=0}^{\left\lfloor \frac{n}{2}\right\rfloor}
            C_{n,j}e_1^{n-2j}e_2^{j+1}\\
            &=0.
        \end{align*}
    Comparing coefficients, we obtain
        \[
            C_{n+2,j+1}-C_{n+1,j+1}+C_{n,j}=0,
            \qquad
            C_{n+2,0}=C_{n+1,0}.
        \]
    In particular,
        \[
            C_{n,0}=1,
            \qquad
            C_{2n,n}=(-1)^n.
        \]

    \begin{prop}
    \label{prop:pascal}
        Under the initial condition $C_{n,0}=1$ and the boundary condition
            \[
                C_{n,j}=0
                \qquad
                (j<0\text{ or }2j>n),
            \]
        the solution of the recurrence
            \[
                C_{n+2,j+1}-C_{n+1,j+1}+C_{n,j}=0
            \]
        is
            \[
                C_{n,j}=(-1)^j\binom{n-j}{j}.
            \]
    \end{prop}

    \begin{proof}
        Set
            \[
                C_{n,j}=(-1)^j\binom{n-j}{j}.
            \]
        For $j=0$, we have
            \[
                C_{n,0}=\binom{n}{0}=1,
            \]
        so the initial condition is satisfied.
        Furthermore,
            \begin{align*}
                C_{n+2,j+1}-C_{n+1,j+1}+C_{n,j}
                &=
                \left(
                    -\frac{n-j+1}{j+1}
                    +\frac{n-2j}{j+1}
                    +1
                \right)
                (-1)^j\binom{n-j}{j}\\
                &=0.
            \end{align*}
        This proves the proposition.
    \end{proof}

    Consequently,
        \[
            H_n(e_1,e_2)
            =
            \sum_{j=0}^{\left\lfloor \frac{n}{2}\right\rfloor}
            (-1)^j
            \binom{n-j}{j}
            e_1^{n-2j}e_2^j.
        \]
    Lemma~\ref{lem:3koukan}, together with the relations between the roots and coefficients of the quadratic equation, now gives the following result.

    \begin{lem}
    \label{lem:3koukan2}
        Let $n\geq2$.
        Suppose that the quadratic equation
            \[
                X^2-pX+q=0
            \]
        over a field $K$ of characteristic zero has no repeated root.
        Then the general solution of the recurrence
            \[
                a_{n+2}-pa_{n+1}+qa_n=0
            \]
        is
            \[
                a_n = H_{n-1}(p,q)a_1 - qH_{n-2}(p,q)a_0.
            \]
    \end{lem}

    \begin{thm}
    \label{thm:T2_HOM}
        For $n\geq2$,
            \[
                P(T(2,n);y,z) = H_{n-1}(y^{-1}z,y^{-2}) - (y^{-1}+y^{-3})z^{-1} H_{n-2}(y^{-1}z,y^{-2}).
            \]
    \end{thm}

    \begin{proof}
        Set
            \[
                a_n:=P(T(2,n);y,z).
            \]
        By Proposition~\ref{prop:unknot_union},
            \[
                a_0 = P(T(2,0);y,z) = P(\bigcirc\cup\bigcirc;y,z) = z^{-1}(y+y^{-1}),
            \]
        while
            \[
                a_1 = P(T(2,1);y,z) = P(\bigcirc;y,z) = 1.
            \]

        Since
            \[
                T(2,n+2)=\overline{\sigma_1^{n+2}},
            \]
        the skein relation gives
            \[
                yP(\overline{\sigma_1^{n+2}};y,z) + y^{-1}P(\overline{\sigma_1^n};y,z) = zP(\overline{\sigma_1^{n+1}};y,z).
            \]
        Hence,
            \[
                a_{n+2} - y^{-1}za_{n+1} + y^{-2}a_n = 0.
            \]

        Let
            \[
                f(X)=X^2-y^{-1}zX+y^{-2}.
            \]
        Its discriminant with respect to $X$ is
            \[
                \operatorname{Disc}_X(f) = y^{-2}z^2-4y^{-2} \neq 0
            \]
        in $\QQ(y,z)$.

        Therefore, for $n\geq2$, Lemma~\ref{lem:3koukan2} gives
            \begin{align*}
                a_n
                &=
                H_{n-1}(y^{-1}z,y^{-2})a_1 - y^{-2}H_{n-2}(y^{-1}z,y^{-2})a_0\\
                &=
                H_{n-1}(y^{-1}z,y^{-2}) - (y^{-1}+y^{-3})z^{-1} H_{n-2}(y^{-1}z,y^{-2}).
            \end{align*}
    \end{proof}

    \begin{ex}
    \label{ex:HOM_trefoil}
        Since
        \begin{align*}
            H_1(y^{-1}z,y^{-2})
                &=y^{-1}z,\\
            H_2(y^{-1}z,y^{-2})
                &=y^{-2}z^2-y^{-2},
        \end{align*}
        the HOMFLY polynomial of $T(2,3)$, that is, the trefoil knot, is
        \begin{align*}
            P(T(2,3);y,z)
            &=
            H_2(y^{-1}z,y^{-2}) - (y^{-1}+y^{-3})z^{-1} H_1(y^{-1}z,y^{-2})\\
            &=
            y^{-2}z^2-2y^{-2}-y^{-4}.
        \end{align*}
    \end{ex}

    Theorems~\ref{thm:T2_HOM} and~\ref{thm:mirror_HOMFLY} give the corresponding formula for negative powers.

    \begin{thm}
    \label{thm:T2_negative_HOM}
        For $n\geq2$,
            \[
                P(T(2,-n);y,z) = H_{n-1}(yz,y^2) - (y^3+y)z^{-1}H_{n-2}(yz,y^2).
            \]
    \end{thm}

    \begin{proof}
        The link $T(2,-n)$ is the mirror image of $T(2,n)$.
        The result therefore follows by applying Theorem~\ref{thm:mirror_HOMFLY} to Theorem~\ref{thm:T2_HOM}.
    \end{proof}
%------------------------------------------------
\subsection{\texorpdfstring{The HOMFLY polynomials of $T(3,n)$}{The HOMFLY polynomials of T(3,n)}}
%------------------------------------------------

    As in the computation of the HOMFLY polynomials of $T(2,n)$, we begin by solving a five-term recurrence.
    Consider the fourth-order linear recurrence
        \[
            x_{n+4}-sx_{n+3}+tx_{n+2}-ux_{n+1}+vx_n=0.
        \]
    A recurrence of this form will arise in the computation of $T(3,n)$.

    Throughout the following discussion, let $K$ be a field of
    characteristic zero.
    Define
        \[
            f(X)=X^4-sX^3+tX^2-uX+v\in K[X].
        \]
    Let $\Omega$ be the splitting field of $f$, and write
        \[
            f(X)=\prod_{i=1}^4(X-\lambda_i),
        \]
    where $\lambda_i\in\Omega$.
    The discriminant of $f$ is defined by
        \[
            \operatorname{Disc}_X(f) = \prod_{1\leq i<j\leq4}(\lambda_i-\lambda_j)^2.
        \]
    We also define the cubic polynomial
        \[
            S(X) = x_0(X^3-sX^2+tX-u) + x_1(X^2-sX+t) + x_2(X-s) + x_3.
        \]

    \begin{lem}
    \label{lem:5koukan_sol}
        Suppose that $\operatorname{Disc}_X(f)\neq0$.
        Then the roots $\lambda_1,\lambda_2,\lambda_3,\lambda_4$ are
        pairwise distinct, and the general solution of
            \[
                x_{n+4}-sx_{n+3}+tx_{n+2}-ux_{n+1}+vx_n=0
            \]
        is
            \[
                x_n = \sum_{i=1}^4 \frac{\lambda_i^nS(\lambda_i)}{f'(\lambda_i)}.
            \]
    \end{lem}

    \begin{proof}
        We perform the following calculations over $\Omega$.
        Let $\{i,j,k,\ell\}=\{1,2,3,4\}$.
        The recurrence can be rewritten as
            \begin{align*}
                &x_{n+4}
                -(\lambda_i+\lambda_j+\lambda_k)x_{n+3}
                +(\lambda_i\lambda_j+\lambda_j\lambda_k
                +\lambda_k\lambda_i)x_{n+2}
                -\lambda_i\lambda_j\lambda_kx_{n+1}\\
                &\quad=
                \lambda_\ell\bigl\{
                    x_{n+3}
                    -(\lambda_i+\lambda_j+\lambda_k)x_{n+2}
                    +(\lambda_i\lambda_j+\lambda_j\lambda_k
                    +\lambda_k\lambda_i)x_{n+1}
                    -\lambda_i\lambda_j\lambda_kx_n
                \bigr\}.
            \end{align*}
        Hence,
            \begin{align*}
                &x_{n+3}
                -(\lambda_i+\lambda_j+\lambda_k)x_{n+2}
                +(\lambda_i\lambda_j+\lambda_j\lambda_k
                +\lambda_k\lambda_i)x_{n+1}
                -\lambda_i\lambda_j\lambda_kx_n\\
                &\quad=
                \lambda_\ell^n\bigl\{
                    x_3
                    -(\lambda_i+\lambda_j+\lambda_k)x_2
                    +(\lambda_i\lambda_j+\lambda_j\lambda_k
                    +\lambda_k\lambda_i)x_1
                    -\lambda_i\lambda_j\lambda_kx_0
                \bigr\}.
            \end{align*}

        For simplicity, set
            \[
                \Lambda_\ell
                :=
                \lambda_\ell^n\bigl\{
                    x_3-(\lambda_i+\lambda_j+\lambda_k)x_2
                    +(\lambda_i\lambda_j+\lambda_j\lambda_k
                    +\lambda_k\lambda_i)x_1
                    -\lambda_i\lambda_j\lambda_kx_0
                \bigr\}.
            \]
        The four equations above can be written as
            \[
            \begin{pmatrix}
            1&-\lambda_2-\lambda_3-\lambda_4&
            \lambda_2\lambda_3+\lambda_3\lambda_4+\lambda_4\lambda_2&
            -\lambda_2\lambda_3\lambda_4\\
            1&-\lambda_1-\lambda_3-\lambda_4&
            \lambda_1\lambda_3+\lambda_3\lambda_4+\lambda_4\lambda_1&
            -\lambda_1\lambda_3\lambda_4\\
            1&-\lambda_1-\lambda_2-\lambda_4&
            \lambda_1\lambda_2+\lambda_2\lambda_4+\lambda_4\lambda_1&
            -\lambda_1\lambda_2\lambda_4\\
            1&-\lambda_1-\lambda_2-\lambda_3&
            \lambda_1\lambda_2+\lambda_2\lambda_3+\lambda_3\lambda_1&
            -\lambda_1\lambda_2\lambda_3
            \end{pmatrix}
            \begin{pmatrix}
                x_{n+3}\\x_{n+2}\\x_{n+1}\\x_n
            \end{pmatrix}
            =
            \begin{pmatrix}
                \Lambda_1\\\Lambda_2\\\Lambda_3\\\Lambda_4
            \end{pmatrix}.
            \]
        Solving this system gives
            \begin{align*}
                x_n
                &=
                \frac{\Lambda_1}
                {(\lambda_1-\lambda_2)(\lambda_1-\lambda_3)
                 (\lambda_1-\lambda_4)}
                +
                \frac{\Lambda_2}
                {(\lambda_2-\lambda_1)(\lambda_2-\lambda_3)
                 (\lambda_2-\lambda_4)}\\
                &\quad+
                \frac{\Lambda_3}
                {(\lambda_3-\lambda_1)(\lambda_3-\lambda_2)
                 (\lambda_3-\lambda_4)}
                +
                \frac{\Lambda_4}
                {(\lambda_4-\lambda_1)(\lambda_4-\lambda_2)
                 (\lambda_4-\lambda_3)}.
            \end{align*}

        By the definition of $S$ and the relations between the roots and coefficients of $f$,
            \begin{align*}
                S(\lambda_i)
                &=
                x_0(\lambda_i^3-s\lambda_i^2+t\lambda_i-u)
                +x_1(\lambda_i^2-s\lambda_i+t)
                +x_2(\lambda_i-s)+x_3\\
                &=
                x_3-(\lambda_j+\lambda_k+\lambda_\ell)x_2
                +(\lambda_j\lambda_k+\lambda_k\lambda_\ell
                +\lambda_\ell\lambda_j)x_1
                -\lambda_j\lambda_k\lambda_\ell x_0.
            \end{align*}
        Thus,
            \[
                \Lambda_i=\lambda_i^nS(\lambda_i).
            \]
        Moreover,
            \[
                f'(\lambda_i) = (\lambda_i-\lambda_j)(\lambda_i-\lambda_k)(\lambda_i-\lambda_\ell).
            \]
        Substitution into the expression above yields
            \[
                x_n = \sum_{i=1}^4 \frac{\lambda_i^nS(\lambda_i)}{f'(\lambda_i)}.
            \]
    \end{proof}

    From now on, assume that $\operatorname{Disc}_X(f)\neq0$.
    Since $x_n$ is symmetric in the roots $\lambda_i$, it can be expressed in terms of the elementary symmetric polynomials $s,t,u,v$.

    Set
        \[
            g_n := \sum_{i=1}^4\frac{\lambda_i^n}{f'(\lambda_i)}.
        \]
    We now express $g_n$ in terms of $s,t,u,v$.
    By partial fraction decomposition,
        \[
            \frac{1}{f(X)} = \frac{1}{X^4-sX^3+tX^2-uX+v} = \sum_{i=1}^4\frac{1}{f'(\lambda_i)(X-\lambda_i)}.
        \]
    Replacing $X$ by $X^{-1}$ and using
        \[
            \frac{1}{1-\lambda_iX}
            =
            \sum_{n=0}^\infty(\lambda_iX)^n
        \]
    in the formal power series ring $\Omega[[X]]$, we obtain
        \begin{align*}
            \frac{X^4}{1-sX+tX^2-uX^3+vX^4}
            &=
            \sum_{i=1}^4
            \frac{X}{f'(\lambda_i)(1-\lambda_iX)}\\
            &=
            \sum_{i=1}^4
            \frac{X}{f'(\lambda_i)}
            \sum_{n=0}^\infty(\lambda_iX)^n\\
            &=
            \sum_{n=0}^\infty
            \left(
                \sum_{i=1}^4\frac{\lambda_i^n}{f'(\lambda_i)}
            \right)X^{n+1}\\
            &=
            \sum_{n=0}^\infty g_nX^{n+1}.
        \end{align*}
    Dividing both sides by $X$ gives
        \[
            \frac{X^3}{1-sX+tX^2-uX^3+vX^4}
            =
            \sum_{n=0}^\infty g_nX^n.
        \]
    Expanding the denominator as a geometric series, we obtain
        \[
            \sum_{n=0}^\infty g_nX^n = X^3\sum_{n=0}^\infty (sX-tX^2+uX^3-vX^4)^n.
        \]
    Since the lowest degree on the right-hand side is $3$, we have
        \[
            g_0=g_1=g_2=0.
        \]

    By the multinomial theorem,
        \begin{align*}
            &(sX-tX^2+uX^3-vX^4)^n\\
            &=
            \sum_{\substack{i+j+k+\ell=n\\i,j,k,\ell\geq0}}
            \frac{n!}{i!j!k!\ell!}
            (sX)^i(-tX^2)^j(uX^3)^k(-vX^4)^\ell\\
            &=
            \sum_{\substack{i+j+k+\ell=n\\i,j,k,\ell\geq0}} \hspace{-3mm} (-1)^{j+\ell} \frac{n!}{i!j!k!\ell!} s^it^ju^kv^\ell X^{i+2j+3k+4\ell}.
        \end{align*}
    Therefore, for $n\geq3$,
        \[ \displaystyle
            g_n = \hspace{-4mm}\sum_{\substack{i+2j+3k+4\ell+3=n\\i,j,k,\ell\geq0}}\hspace{-4mm} (-1)^{j+\ell} \frac{(i+j+k+\ell)!}{i!j!k!\ell!} s^it^ju^kv^\ell.
        \]

    Since
        \begin{align*}
            S(X)
            &=
            x_0(X^3-sX^2+tX-u)
            +x_1(X^2-sX+t)+x_2(X-s)+x_3\\
            &=
            x_0X^3-(sx_0-x_1)X^2
            +(tx_0-sx_1+x_2)X
            -(ux_0-tx_1+sx_2-x_3),
        \end{align*}
    substituting this expression into Lemma~\ref{lem:5koukan_sol} gives the following lemma.

    \begin{lem}
    \label{lem:5koukan}
        Suppose that $\operatorname{Disc}_X(f)\neq0$.
        A solution of
            \[
                x_{n+4}-sx_{n+3}+tx_{n+2}-ux_{n+1}+vx_n=0
            \]
        is given by
            \[
                x_n = x_0g_{n+3} - (sx_0-x_1)g_{n+2} + (tx_0-sx_1+x_2)g_{n+1} - (ux_0-tx_1+sx_2-x_3)g_n.
            \]
    \end{lem}

    We now compute the HOMFLY polynomial of
        \[
            T(3,n)=\overline{(\sigma_1\sigma_2)^n}.
        \]
    When it is necessary to emphasize the number of strands, we denote the closure of $\beta\in B_n$ by $\overline{\beta}\,_n$.

    Set
        \[
            a_n
            :=
            P(\overline{(\sigma_1\sigma_2)^n\sigma_1};y,z),
            \qquad
            b_n
            :=
            P(\overline{(\sigma_1\sigma_2)^n\sigma_2};y,z),
            \qquad
            c_n
            :=
            P(T(3,n);y,z).
        \]
    The sequences $a_n$ and $b_n$ are auxiliary sequences obtained by
    adding one positive crossing to $T(3,n)$, while $c_n$ is the sequence that we wish to determine.

    \begin{prop}
    \label{prop:initial_an}
        \[
        \begin{gathered}
            a_0=(y+y^{-1})z^{-1},\qquad
            a_1=P(T(2,2);y,z),\qquad
            a_2=P(T(2,4);y,z),\\
            a_3
            =
            \{-y^{-3}z+(y^{-3}+y^{-5})z^{-1}\}
            P(T(2,3);y,z)
            +
            y^{-1}zP(T(2,5);y,z).
        \end{gathered}
        \]
    \end{prop}

    \begin{proof}
        For $n=0,1,2$, Lemma~\ref{lem:cyclic_closure}, Proposition~\ref{prop:unknot_union}, Theorem~\ref{thm:Markov}(2), and the braid relation give
            \begin{align*}
                a_0
                &=
                P(\overline{\sigma_1}_3;y,z)\\
                &=
                P(\bigcirc\cup\bigcirc;y,z)\\
                &=
                (y+y^{-1})z^{-1},
            \end{align*}
            \begin{align*}
                a_1
                &=
                P(\overline{\sigma_1\sigma_2\sigma_1};y,z)\\
                &=
                P(\overline{\sigma_1\sigma_1\sigma_2};y,z)\\
                &=
                P(\overline{\sigma_1^2}_2;y,z)\\
                &=
                P(T(2,2);y,z),
            \end{align*}
        and
            \begin{align*}
                a_2
                &=
                P(\overline{\sigma_1\sigma_2\sigma_1\sigma_2\sigma_1};y,z)\\
                &=
                P(\overline{\sigma_1\sigma_1\sigma_2\sigma_1\sigma_1};y,z)\\
                &=
                P(\overline{\sigma_1^4\sigma_2};y,z)\\
                &=
                P(\overline{\sigma_1^4}_2;y,z)\\
                &=
                P(T(2,4);y,z).
            \end{align*}
        In the last steps for $a_1$ and $a_2$, we use Markov destabilization, the inverse operation of Theorem~\ref{thm:Markov}(2).

        For $n=3$, we first write
            \[
                (\sigma_1\sigma_2)^3\sigma_1
                =
                \sigma_1\sigma_2\sigma_1\sigma_2\sigma_1\sigma_2\sigma_1.
            \]
        By applying Lemma~\ref{lem:cyclic_closure}, we may move the last $\sigma_1$ to the front after taking the closure. Hence
            \[
                \overline{(\sigma_1\sigma_2)^3\sigma_1} \cong \overline{\sigma_1^2\sigma_2\sigma_1\sigma_2\sigma_1\sigma_2}.
            \]
        Using the braid relation
            \[
                \sigma_2\sigma_1\sigma_2=\sigma_1\sigma_2\sigma_1,
            \]
        we obtain
            \[
                \overline{\sigma_1^2\sigma_2\sigma_1\sigma_2\sigma_1\sigma_2} = \overline{\sigma_1^3\sigma_2\sigma_1^2\sigma_2}.
            \]
        Thus,
            \[
                \overline{(\sigma_1\sigma_2)^3\sigma_1} \cong \overline{\sigma_1^3\sigma_2\sigma_1^2\sigma_2}.
            \]
        Applying the skein relation to either of the two crossings corresponding to the factor $\sigma_1^2$, we obtain
            \begin{align*}
                ya_3
                +y^{-1}P(\overline{\sigma_1^3\sigma_2^2};y,z)
                &=
                zP(\overline{\sigma_1^3\sigma_2\sigma_1\sigma_2};y,z)\\
                &=
                zP(\overline{\sigma_1^5\sigma_2};y,z)\\
                &=
                zP(\overline{\sigma_1^5}_2;y,z)\\
                &=
                zP(T(2,5);y,z).
            \end{align*}
        Hence,
            \[
                a_3 = -y^{-2}P(\overline{\sigma_1^3\sigma_2^2};y,z) + y^{-1}zP(T(2,5);y,z).
            \]
        The passage to a two-strand braid again uses Markov destabilization.

        Furthermore, Proposition~\ref{prop:unknot_union} gives
            \begin{align*}
                &yP(\overline{\sigma_1^3\sigma_2^2};y,z)
                +y^{-1}P(\overline{\sigma_1^3}_3;y,z)
                =
                zP(\overline{\sigma_1^3\sigma_2};y,z),\\
                &yP(\overline{\sigma_1^3\sigma_2^2};y,z)
                +y^{-1}P(\bigcirc\cup\overline{\sigma_1^3}_2;y,z)
                =
                zP(\overline{\sigma_1^3}_2;y,z),\\
                &yP(\overline{\sigma_1^3\sigma_2^2};y,z)
                +y^{-1}z^{-1}(y+y^{-1})P(T(2,3);y,z)
                =
                zP(T(2,3);y,z).
            \end{align*}
        Therefore,
            \[
                P(\overline{\sigma_1^3\sigma_2^2};y,z) = \{y^{-1}z-(y^{-1}+y^{-3})z^{-1}\} P(T(2,3);y,z).
            \]
        Substitution yields
            \[
                a_3 = \{-y^{-3}z+(y^{-3}+y^{-5})z^{-1}\} P(T(2,3);y,z) + y^{-1}zP(T(2,5);y,z).
            \]
    \end{proof}

        The results of the preceding subsection give the following values.

    \begin{prop}
    \label{prop:T2n}
        \begin{align*}
            P(T(2,2);y,z)
                &=y^{-1}z-(y^{-1}+y^{-3})z^{-1},\\
            P(T(2,3);y,z)
                &=y^{-2}z^2-2y^{-2}-y^{-4},\\
            P(T(2,4);y,z)
                &=y^{-3}z^3-(3y^{-3}+y^{-5})z
                  +(y^{-3}+y^{-5})z^{-1},\\
            P(T(2,5);y,z)
                &=y^{-4}z^4-(4y^{-4}+y^{-6})z^2
                  +3y^{-4}+2y^{-6}.
        \end{align*}
    \end{prop}

    \begin{proof}
        The value of $P(T(2,3);y,z)$ was obtained in
        Example~\ref{ex:HOM_trefoil}.
        By Theorem~\ref{thm:T2_HOM},
            \begin{align*}
                P(T(2,2);y,z)
                    &=
                    H_1(y^{-1}z,y^{-2})
                    -(y^{-1}+y^{-3})z^{-1}H_0(y^{-1}z,y^{-2}),\\
                P(T(2,4);y,z)
                    &=
                    H_3(y^{-1}z,y^{-2})
                    -(y^{-1}+y^{-3})z^{-1}H_2(y^{-1}z,y^{-2}),\\
                P(T(2,5);y,z)
                    &=
                    H_4(y^{-1}z,y^{-2})
                    -(y^{-1}+y^{-3})z^{-1}H_3(y^{-1}z,y^{-2}).
            \end{align*}
        Here,
            \begin{align*}
                H_4(y^{-1}z,y^{-2})
                    &=y^{-4}z^4-3y^{-4}z^2+y^{-4},\\
                H_3(y^{-1}z,y^{-2})
                    &=y^{-3}z^3-2y^{-3}z,\\
                H_2(y^{-1}z,y^{-2})
                    &=y^{-2}z^2-y^{-2},\\
                H_1(y^{-1}z,y^{-2})
                    &=y^{-1}z,\\
                H_0(y^{-1}z,y^{-2})
                    &=1.
            \end{align*}
        Substituting these expressions gives the desired formulas.
    \end{proof}

    \begin{lem}
    \label{lem:anbncn}
        The sequences $a_n,b_n,c_n$ defined above satisfy the  following:
            \begin{enumerate}
                \setlength{\itemsep}{5pt}
                \item $a_n=b_n$ for every $n\in\ZZ_{\geq0}$.
                \item $ya_n+y^{-1}a_{n-1}=zc_n$ for every $n\in\NN$.
            \end{enumerate}
    \end{lem}

    \begin{proof}
        \begin{enumerate}
        \item
        For $n=0$,
            \[
                \overline{\sigma_1}_3
                =
                \bigcirc\cup\bigcirc
                =
                \overline{\sigma_2}_3,
            \]
        and hence $a_0=b_0$.

        Now suppose that $n\geq1$.
        By Lemma~\ref{lem:cyclic_closure},
            \[
                P(\overline{(\sigma_1\sigma_2)^n\sigma_1^{-1}};y,z)
                =
                P(\overline{\sigma_2(\sigma_1\sigma_2)^{n-1}};y,z)
                =
                P(\overline{(\sigma_1\sigma_2)^{n-1}\sigma_2};y,z)
                =
                b_{n-1},
            \]
        and
            \[
                P(\overline{(\sigma_1\sigma_2)^n\sigma_2^{-1}};y,z)
                =
                P(\overline{(\sigma_1\sigma_2)^{n-1}\sigma_1};y,z)
                =
                a_{n-1}.
            \]
        Applying the skein relation gives
            \begin{align*}
                yP(\overline{(\sigma_1\sigma_2)^n\sigma_1};y,z)
                +y^{-1}P(\overline{(\sigma_1\sigma_2)^n\sigma_1^{-1}};y,z)
                &=
                zP(\overline{(\sigma_1\sigma_2)^n};y,z),\\
                yP(\overline{(\sigma_1\sigma_2)^n\sigma_2};y,z)
                +y^{-1}P(\overline{(\sigma_1\sigma_2)^n\sigma_2^{-1}};y,z)
                &=
                zP(\overline{(\sigma_1\sigma_2)^n};y,z).
            \end{align*}
        Thus,
            \[
                ya_n+y^{-1}b_{n-1}=zc_n,
                \qquad
                yb_n+y^{-1}a_{n-1}=zc_n.
                \tag{$*$}
            \]
        Subtracting the second equation from the first gives
            \[
                a_n-b_n
                =
                y^{-2}(a_{n-1}-b_{n-1}),
            \]
        and hence
            \[
                a_n-b_n
                =
                y^{-2n+2}(a_1-b_1).
            \]
        Using the braid relation followed by Lemma~\ref{lem:cyclic_closure}, we have
            \[
                \overline{\sigma_1\sigma_2\sigma_1}
                =
                \overline{\sigma_2\sigma_1\sigma_2}
                =
                \overline{\sigma_1\sigma_2\sigma_2}.
            \]
        Therefore,
            \[
                a_1
                =
                P(\overline{\sigma_1\sigma_2\sigma_1};y,z)
                =
                P(\overline{\sigma_1\sigma_2\sigma_2};y,z)
                =
                b_1.
            \]
        Thus, $a_n=b_n$ for every $n\in\ZZ_{\geq0}$.

        \item
        This follows immediately from part (1) and~$(*)$.
        \end{enumerate}
    \end{proof}

    \begin{lem}
    \label{lem:an5koukan}
        For every $n\geq0$,
            \[
                a_{n+4} - (y^{-2}z^2-y^{-2}) a_{n+3} + y^{-4}z^2a_{n+2} - (y^{-6}z^2-y^{-6}) a_{n+1} + y^{-8}a_n = 0.
            \]
    \end{lem}

    \begin{proof}
        We prove an equivalent relation with the index shifted by two.
        
        Suppose that $n\geq2$.
        The braid relation gives
            \[
                T(3,n+2) = \overline{(\sigma_1\sigma_2)^n \sigma_1\sigma_2\sigma_1\sigma_2} = \overline{(\sigma_1\sigma_2)^n \sigma_1\sigma_1\sigma_2\sigma_1}.
            \]
        Applying the skein relation, we obtain
            \[
                yc_{n+2} + y^{-1} P(\overline{(\sigma_1\sigma_2)^n\sigma_2\sigma_1};y,z) = za_{n+1}.
            \]
        Applying the skein relation once more gives
            \begin{align*}
                &P(\overline{(\sigma_1\sigma_2)^n\sigma_2\sigma_1};y,z)\\
                &\quad=
                -y^{-2}
                P(\overline{(\sigma_1\sigma_2)^{n-1}\sigma_1^2};y,z)
                +
                y^{-1}za_n,
            \end{align*}
        and
            \[
                P(\overline{(\sigma_1\sigma_2)^{n-1}\sigma_1^2};y,z) = -y^{-2}c_{n-1} + y^{-1}za_{n-1}.
            \]
        Consequently,
            \begin{align*}
                yc_{n+2}
                &+
                y^{-1} P(\overline{(\sigma_1\sigma_2)^n\sigma_2\sigma_1};y,z)\\
                &=
                yc_{n+2}+y^{-5}c_{n-1}
                -y^{-4}za_{n-1}+y^{-2}za_n\\
                &=
                za_{n+1}.
            \end{align*}
        By Lemma~\ref{lem:anbncn},
            \[
                c_n = yz^{-1}a_n+y^{-1}z^{-1}a_{n-1}.
            \]
        Substituting this expression yields
            \begin{align*}
                0={}&
                y^2z^{-1}a_{n+2} - (z-z^{-1})a_{n+1} + y^{-2}za_n\\
                &\quad
                - y^{-4}(z-z^{-1})a_{n-1} +y^{-6}z^{-1}a_{n-2}.
            \end{align*}
        Setting $m=n-2$ gives the stated recurrence for every $m\geq0$.
    \end{proof}

    Define
        \[
            f_{y,z}(X)
            =
            X^4-(y^{-2}z^2-y^{-2})X^3
            +y^{-4}z^2X^2
            -(y^{-6}z^2-y^{-6})X
            +y^{-8}
            \in\QQ(y,z)[X].
        \]
    Let $\Omega$ be the splitting field of $f_{y,z}$ and write
        \[
            f_{y,z}(X)=\prod_{i=1}^4(X-\lambda_i).
        \]

    \begin{lem}
    \label{lem:disc}
        The discriminant of $f_{y,z}$ with respect to $X$ is nonzero:
            \[
                \operatorname{Disc}_X(f_{y,z})\neq0
                \qquad\text{in }\QQ(y,z).
            \]
    \end{lem}

    \begin{proof}
        Put
            \[
                s=y^{-2}z^2-y^{-2},
                \qquad
                t=y^{-4}z^2,
                \qquad
                u=y^{-6}z^2-y^{-6},
                \qquad
                v=y^{-8}.
            \]
        Since $f_{y,z}$ is monic, direct calculation gives
            \begin{align*}
                \operatorname{Disc}_X(f_{y,z})
                &=
                \operatorname{Res}(f_{y,z},f'_{y,z})\\
                &=
                -3y^{-24}z^2(z^2-4)(z^2-3)^4
                \neq0.
            \end{align*}
    \end{proof}

    To apply Lemma~\ref{lem:5koukan} to the recurrence for $a_n$, define $G_n$ by
        \[
            \frac{X^3}{1-(y^{-2}z^2-y^{-2})X+y^{-4}z^2X^2 - (y^{-6}z^2-y^{-6})X^3+y^{-8}X^4} = \sum_{n=0}^\infty G_nX^n.
        \]
    This is the sequence $g_n$ in Lemma~\ref{lem:5koukan} specialized to
        \[
            s=y^{-2}z^2-y^{-2},
            \qquad
            t=y^{-4}z^2,
            \qquad
            u=y^{-6}z^2-y^{-6},
            \qquad
            v=y^{-8}.
        \]
    Hence, $G_n$ satisfies the same fourth-order recurrence and has initial
    values
        \[
            G_0=G_1=G_2=0,
            \qquad
            G_3=1.
        \]

    Under the condition $i+2j+3k+4\ell+3=n$, we have
        \begin{align*}
            s^it^ju^kv^\ell
            &=
            y^{-2i-4j-6k-8\ell}
            z^{2j}(z^2-1)^{i+k}\\
            &=
            y^{-2(n-3)}z^{2j}(z^2-1)^{i+k}.
        \end{align*}
    Moreover,
        \[
            z^{2j}(z^2-1)^{i+k} = z^{2(i+j+k)} \sum_{r=0}^{i+k} (-1)^r\binom{i+k}{r}z^{-2r}.
        \]
    Thus, for $n\geq3$,
        \[ \displaystyle
            G_n = \hspace{-4mm}
            \sum_{\substack{i+2j+3k+4\ell+3=n\\i,j,k,\ell\geq0}} \hspace{-6mm}
            (-1)^{j+\ell} \frac{(i+j+k+\ell)!}{i!j!k!\ell!} y^{-2(n-3)}z^{2(i+j+k)} \sum_{r=0}^{i+k} (-1)^r\binom{i+k}{r}z^{-2r}.
        \]

    Using Propositions~\ref{prop:initial_an} and~\ref{prop:T2n}, set
        \begin{align*}
            C_3
                &:=
                a_0
                =
                (y+y^{-1})z^{-1},\\
            C_2
                &:=
                y^{-2}(z^2-1)a_0-a_1
                =
                y^{-3}z,\\
            C_1
                &:=
                y^{-4}z^2a_0-y^{-2}(z^2-1)a_1+a_2
                =
                y^{-5}z,\\
            C_0
                &:=
                y^{-6}(z^2-1)a_0-y^{-4}z^2a_1
                +y^{-2}(z^2-1)a_2-a_3\\
                &=
                (y^{-7}+y^{-9})z^{-1}.
        \end{align*}
    Lemmas~\ref{lem:5koukan}, \ref{lem:an5koukan}, and~\ref{lem:disc}
    now imply
        \[
            a_n = C_3G_{n+3}-C_2G_{n+2}+C_1G_{n+1}-C_0G_n.
        \]

    By Lemma~\ref{lem:anbncn},
        \begin{align*}
            c_n
            &=
            yz^{-1}a_n+y^{-1}z^{-1}a_{n-1}\\
            &=
            yz^{-1}C_3G_{n+3}
            +(-yz^{-1}C_2+y^{-1}z^{-1}C_3)G_{n+2}\\
            &\quad+
            (yz^{-1}C_1-y^{-1}z^{-1}C_2)G_{n+1}\\
            &\quad+
            (-yz^{-1}C_0+y^{-1}z^{-1}C_1)G_n
            -y^{-1}z^{-1}C_0G_{n-1}.
        \end{align*}

    Denote the coefficient of $G_{n+i}$ by $D_i$ for
    $i=3,2,1,0,-1$. Then
        \begin{align*}
            D_3
                &=yz^{-1}C_3
                =(y^2+1)z^{-2},\\
            D_2
                &=-yz^{-1}C_2+y^{-1}z^{-1}C_3\\
                &=-y^{-2}+(1+y^{-2})z^{-2},\\
            D_1
                &=yz^{-1}C_1-y^{-1}z^{-1}C_2
                =0,\\
            D_0
                &=-yz^{-1}C_0+y^{-1}z^{-1}C_1\\
                &=-(y^{-6}+y^{-8})z^{-2}+y^{-6},\\
            D_{-1}
                &=-y^{-1}z^{-1}C_0\\
                &=-(y^{-8}+y^{-10})z^{-2}.
        \end{align*}

    We have therefore proved the following theorem.

    \begin{thm}
    \label{thm:T3_HOM}
        For every $n\geq1$,
        \begin{align*}
            P(T(3,n);y,z)
            &=
            (y^2+1)z^{-2}G_{n+3} + \bigl((1+y^{-2})z^{-2}-y^{-2}\bigr)G_{n+2}\\
            &\quad+ \bigl(y^{-6}-(y^{-6}+y^{-8})z^{-2}\bigr)G_n - (y^{-8}+y^{-10})z^{-2}G_{n-1}.
        \end{align*}
    \end{thm}

    \begin{ex}
    \label{ex:T32}
        For $n=2$, Theorem~\ref{thm:T3_HOM} gives
            \[
                P(T(3,2);y,z) = y^{-2}z^2-2y^{-2}-y^{-4}.
            \]
        By Theorem~\ref{thm:torus_symmetry}, $T(3,2)\cong T(2,3)$.
        Thus, this agrees with the HOMFLY polynomial of $T(2,3)$ obtained in Example~\ref{ex:HOM_trefoil}.
    \end{ex}

    \begin{ex}
    \label{ex:T33}
        For $n=3$, Theorem~\ref{thm:T3_HOM} gives
            \[
            \begin{aligned}
                P(T(3,3);y,z)
                ={}&
                y^{-4}z^4-4y^{-4}z^2+3y^{-4}-y^{-4}z^{-2}\\
                &{}-y^{-6}z^2+3y^{-6}
                -2y^{-6}z^{-2}-y^{-8}z^{-2}.
            \end{aligned}
            \]
    \end{ex}

    \begin{ex}
        For $n=4$, Theorem~\ref{thm:T3_HOM} gives
            \[
            \begin{aligned}
                P(T(3,4);y,z)
                ={}&
                y^{-6}z^6-6y^{-6}z^4+10y^{-6}z^2-5y^{-6}\\
                &{}-y^{-8}z^4+5y^{-8}z^2-5y^{-8}-y^{-10}.
            \end{aligned}
            \]
    \end{ex}

%-------------------------------------------------
\subsection{\texorpdfstring{The HOMFLY polynomials of $T(-3,n)$}{The HOMFLY polynomials of T(-3,n)}}
%-------------------------------------------------

    We now obtain the formula for $T(-3,n)$ using mirror images.

    Define
        \[
            \widetilde{G}_n := \left.G_n\right|_{y\mapsto y^{-1}}.
        \]
    Then
        \[
            \widetilde{G}_0 = \widetilde{G}_1 = \widetilde{G}_2 = 0.
        \]
    For $n\geq3$, the explicit formula for $G_n$ gives
        \[\displaystyle
            \widetilde{G}_n = \hspace{-4mm} \sum_{\substack{i+2j+3k+4\ell+3=n\\i,j,k,\ell\geq0}} \hspace{-5mm} (-1)^{j+\ell} \frac{(i+j+k+\ell)!}{i!j!k!\ell!} y^{2(n-3)}z^{2(i+j+k)} \sum_{r=0}^{i+k} (-1)^r\binom{i+k}{r}z^{-2r}.
        \]
    
    \begin{cor}
    \label{cor:T_minus3_HOM}
        For every $n\geq1$,
        \begin{align*}
            P(T(-3,n);y,z)
            &=
            (y^{-2}+1)z^{-2}\widetilde{G}_{n+3} + \bigl((1+y^2)z^{-2}-y^2\bigr) \widetilde{G}_{n+2}\\
            &\quad+
            \bigl(y^6-(y^6+y^8)z^{-2}\bigr) \widetilde{G}_n - (y^8+y^{10})z^{-2}\widetilde{G}_{n-1}.
        \end{align*}
    \end{cor}

    \begin{proof}
        Apply Theorem~\ref{thm:mirror_HOMFLY} to Theorem~\ref{thm:T3_HOM}.
    \end{proof}

    \begin{ex}
        For $n=2$, Corollary~\ref{cor:T_minus3_HOM} gives
            \[
                P(T(-3,2);y,z)
                =
                y^2z^2-2y^2-y^4.
            \]
        By Example~\ref{ex:T32},
            \[
                P(T(-3,2);y,z)
                \neq
                P(T(3,2);y,z).
            \]
        Therefore,
            \[
                T(-3,2)\not\cong T(3,2).
            \]
        This recovers the well-known fact that the trefoil knot is not equivalent to its mirror image; in other words, the trefoil is chiral.
    \end{ex}

    \begin{ex}
        For $n=3$, Corollary~\ref{cor:T_minus3_HOM} gives
            \[
            \begin{aligned}
                P(T(-3,3);y,z)
                ={}&
                y^4z^4-4y^4z^2+3y^4-y^4z^{-2}\\
                &{}-y^6z^2+3y^6-2y^6z^{-2}-y^8z^{-2}.
            \end{aligned}
            \]
        By Example~\ref{ex:T33},
            \[
                P(T(-3,3);y,z)
                \neq
                P(T(3,3);y,z).
            \]
        Hence, with the prescribed orientations,
            \[
                T(-3,3)\not\cong T(3,3).
            \]
    \end{ex}
%====================================================
\section{Applications}
\label{sec:applications}
%====================================================

    In this section, we derive several consequences concerning the torus links $T(3,n)$ from the explicit formula for their HOMFLY polynomials obtained in the previous section.
    In particular, by considering the highest-degree term with respect to $z$, we show that the links $T(3,n)$ are mutually distinct.
    We also see, by using the mirror-image formula, that $T(3,n)$ is not equivalent to its mirror image $T(-3,n)$ in general.

%----------------------------------------------------
\subsection{\texorpdfstring{The highest-degree term of $P(T(3,n);y,z)$}{The highest-degree term of P(T(3,n);y,z)}}
%----------------------------------------------------

    We first determine the highest-degree term of $P(T(3,n);y,z)$ with respect to $z$ from the explicit formula obtained in the previous section.

    \begin{lem}
    \label{lem:highest_degz_Gn}
        For $n\geq 3$, we have
            \[
                \deg_z G_n = 2(n-3),
            \]
        and its highest-degree term with respect to $z$ is
            \[
                y^{-2(n-3)}z^{2(n-3)}.
            \]
    \end{lem}
    
    \begin{proof}
        For $n\geq 3$, the polynomial $G_n$ is given by
            \[
                \displaystyle G_n =
                \hspace{-5mm}\sum_{\substack{i+2j+3k+4\ell + 3 = n\\i,j,k,\ell \geq 0}}\hspace{-5mm}
                (-1)^{j+\ell}\frac{(i+j+k+\ell)!}{i!j!k!\ell!}
                y^{-2(n-3)}z^{2(i+j+k)}
                \sum_{r=0}^{i+k} (-1)^r \binom{i+k}{r}z^{-2r}.
            \]
        Hence each term has degree
            \[
                2(i+j+k)-2r
            \]
        with respect to $z$.
    
        In order to attain the highest degree, we must have $r=0$.
        Moreover, since
            \[
                i+j+k \leq i + 2j + 3k + 4 \ell = n-3,
            \]
        equality holds only when
            \[
                j+2k+4\ell = 0,
            \]
        that is, only when
            \[
                j=k=\ell =0.
            \]
        
        Therefore, the highest degree with respect to $z$ is $2(n-3)$, and the term attaining this degree is only the term corresponding to
            \[
                (i,j,k,\ell,r)=(n-3,0,0,0,0).
            \]

        Furthermore,
            \[
                \displaystyle (-1)^0 \frac{(n-3)!}{(n-3)!0!0!0!} \cdot (-1)^0 \binom{n-3}{0} = 1.
            \]
        Thus the highest-degree term with respect to $z$ is
            \[
                y^{-2(n-3)}z^{2(n-3)}.
            \]
    \end{proof}
    
    Using this lemma, we obtain the highest-degree term of $P(T(3,n);y,z)$ as follows.
    
    \begin{prop}
    \label{prop:highest_degz_T3n}
        For any $n\geq 1$, we have
            \[
                \deg_z P(T(3,n);y,z) = 2n-2,
            \]
        and its highest-degree term with respect to $z$ is
            \[
                y^{-2n+2}z^{2n-2}.
            \]
    \end{prop}
    
    \begin{proof}
        By Theorem~\ref{thm:T3_HOM}, we have
            \[
            \begin{aligned}
                P(T(3,n);y,z)
                ={}& (y^2+1)z^{-2}G_{n+3}
                +((1+y^{-2})z^{-2}-y^{-2})G_{n+2} \\
                &\quad +(y^{-6}-(y^{-6}+y^{-8})z^{-2})G_n
                -(y^{-8}+y^{-10})z^{-2}G_{n-1}.
            \end{aligned}
            \]

        If $n=1$, then $T(3,1)$ is the trivial knot, and hence
            \[
                P(T(3,1);y,z) = 1.
            \]
        Thus the assertion holds in this case.
        The case $n=2$ follows from Example~\ref{ex:T32}.

        We now assume that $n \geq 3$.
        By Lemma~\ref{lem:highest_degz_Gn}, the highest-degree term of $(y^2+1)z^{-2} G_{n+3}$ with respect to $z$ is
            \[
                (y^2+1)y^{-2n}z^{2n-2}.
            \]

        Also, the highest-degree term of $G_{n+2}$ with respect to $z$ is
            \[
                y^{-2(n-1)}z^{2(n-1)}.
            \]
        The summand $(1+y^{-2})z^{-2}G_{n+2}$ has degree at most $2n-4$ with respect to $z$, whereas the summand $-y^{-2}G_{n+2}$ contributes
            \[
                -y^{-2}y^{-2n+2}z^{2n-2} = -y^{-2n}z^{2n-2}.
            \]
        Therefore, the highest-degree term of $P(T(3,n);y,z)$ with respect to $z$ is
            \[
                (y^2+1)y^{-2n}z^{2n-2} -y^{-2n}z^{2n-2}
                = y^{-2n+2}z^{2n-2}.
            \]
    
        On the other hand, all the remaining summands have degree lower than $2n-2$ with respect to $z$.
        This proves the assertion.
    \end{proof}

%----------------------------------------------------
\subsection{\texorpdfstring{Distinguishing the links $T(3,n)$}{Distinguishing the links T(3,n)}}
%----------------------------------------------------

    It is already known from the classification of torus links that torus links with different parameters are distinct in general.
    Here we observe that, in the present family, this distinction can be detected directly from the explicit formula for the HOMFLY polynomial obtained in this paper.
    By Proposition~\ref{prop:highest_degz_T3n}, the torus links $T(3,n)$ are mutually distinct.
    
    \begin{cor}
    \label{cor:T3n_distinct}
        Let $m$ and $n$ be positive integers.
        If $m\neq n$, then
            \[
                T(3,m)\not\cong T(3,n).
            \]
    \end{cor}
    
    \begin{proof}
        By Proposition~\ref{prop:highest_degz_T3n}, we have
            \[
                \deg_z P(T(3,n);y,z)=2n-2.
            \]
        In particular, the case in which one of the parameters is equal to $1$ is included, since $P(T(3,1);y,z)=1$ has $z$-degree $0$, while the degree is positive for $T(3,k)$ with $k\geq2$.
        Therefore, if $m\neq n$, then
            \[
                \deg_z P(T(3,m);y,z)=2m-2
                \neq
                2n-2=\deg_z P(T(3,n);y,z).
            \]
        Hence
            \[
                P(T(3,m);y,z) \neq P(T(3,n);y,z).
            \]
    
        Since the HOMFLY polynomial is an invariant of oriented links, it follows that $T(3,m)$ and $T(3,n)$ are not equivalent.
    \end{proof}

%----------------------------------------------------
\subsection{Comparison with mirror images}
%----------------------------------------------------

    We next compare $T(3,n)$ with its mirror image $T(-3,n)$.
    
    \begin{cor}
    \label{cor:T3n_mirror_distinct}
        For $n\geq 2$, we have
            \[
                T(3,n)\not\cong T(-3,n).
            \]
    \end{cor}
    
    \begin{proof}
        By Proposition~\ref{prop:highest_degz_T3n}, the highest-degree term of $P(T(3,n);y,z)$ with respect to $z$ is
            \[
                y^{-2n+2}z^{2n-2}.
            \]
        On the other hand, by Theorem~\ref{thm:mirror_HOMFLY} and Proposition~\ref{prop:highest_degz_T3n}, the highest-degree term of $P(T(-3,n);y,z)$ with respect to $z$ is
            \[
                y^{2n-2}z^{2n-2}.
            \]
        
        If $n\geq 2$, then
            \[
                y^{-2n+2} \neq y^{2n-2}.
            \]
        Hence
            \[
                P(T(3,n);y,z)\neq P(T(-3,n);y,z).
            \]
        Since the HOMFLY polynomial is an invariant of oriented links, it follows that
            \[
                T(3,n)\not\cong T(-3,n).
            \]
    \end{proof}

    \begin{rem}
        Although the HOMFLY polynomial is not a complete invariant for links in general, it may distinguish all members of a restricted family.
        The above computation shows that it distinguishes all members of the family $\{T(3,n)\}_{n\geq 1}$, because
            \[
                n=\frac{\deg_z P(T(3,n);y,z)+2}{2}.
            \]
        Hence the HOMFLY polynomial determines the parameter $n$ within this family.
    \end{rem}

\bibliographystyle{unsrt}
\bibliography{biblio}

\begin{thebibliography}{10}

\bibitem{Kawauchi2015KnotTheory}
Akio Kawauchi.
\newblock {\em Knot Theory}, volume~10 of {\em Kyoritsu Lecture Series:
  Mathematical Adventure}.
\newblock Kyoritsu Shuppan, Tokyo, 2015.

\bibitem{FreydYetterHosteLickorishMillettOcneanu1985}
Peter Freyd, David Yetter, Jim Hoste, W.~B.~Raymond Lickorish, Kenneth Millett,
  and Adrian Ocneanu.
\newblock A new polynomial invariant of knots and links.
\newblock {\em Bulletin of the American Mathematical Society}, 12(2):239--246,
  1985.

\bibitem{PrzytyckiTraczyk1987}
J{\'o}zef~H. Przytycki and Pawe{\l} Traczyk.
\newblock Invariants of links of {C}onway type.
\newblock {\em Kobe Journal of Mathematics}, 4(2):115--139, 1987.

\bibitem{Jones1987Hecke}
Vaughan F.~R. Jones.
\newblock Hecke algebra representations of braid groups and link polynomials.
\newblock {\em Annals of Mathematics}, 126(2):335--388, 1987.

\bibitem{RossoJones1993}
Marc Rosso and Vaughan F.~R. Jones.
\newblock On the invariants of torus knots derived from quantum groups.
\newblock {\em Journal of Knot Theory and Its Ramifications}, 2(1):97--112,
  1993.

\bibitem{LinZheng2010}
Xiao-Song Lin and Hao Zheng.
\newblock On the {H}ecke algebras and the colored {HOMFLY} polynomial.
\newblock {\em Transactions of the American Mathematical Society},
  362(1):1--18, 2010.

\bibitem{BriniEynardMarino2012}
Andrea Brini, Bertrand Eynard, and Marcos Mari{\~n}o.
\newblock Torus knots and mirror symmetry.
\newblock {\em Annales Henri Poincar{\'e}}, 13:1873--1910, 2012.

\bibitem{Birman1974BraidsLinksMCG}
Joan~S. Birman.
\newblock {\em Braids, Links, and Mapping Class Groups}.
\newblock Number~82 in Annals of Mathematics Studies. Princeton University
  Press, Princeton, NJ, 1974.
\newblock Based on lecture notes by James Cannon.

\bibitem{Markov1936FreeEquivalenceClosedBraids}
A.~Markoff.
\newblock \"uber die freie \"aquivalenz der geschlossenen z\"opfe.
\newblock {\em Rec. Math. [Mat. Sbornik] N.S.}, 1(1):73--78, 1936.
\newblock Received 1935-10-07. (German).

\bibitem{Lickorish1997Introduction}
W.~B.~Raymond Lickorish.
\newblock {\em An Introduction to Knot Theory}, volume 175 of {\em Graduate
  Texts in Mathematics}.
\newblock Springer, 1997.

\bibitem{Adams2004KnotBook}
Colin~C. Adams.
\newblock {\em The Knot Book: An Elementary Introduction to the Mathematical
  Theory of Knots}.
\newblock American Mathematical Society, Providence, RI, 2004.

\end{thebibliography}

\end{document}